\documentclass[11pt]{article}

\usepackage{enumitem}
\usepackage{amsmath, amssymb, amsthm}
\usepackage{physics}
\usepackage{graphicx}
\usepackage{subcaption}
\usepackage{xcolor}
\usepackage{tikz}
\usepackage{hyperref}
\usepackage{algorithm, algorithmic}
\usepackage{array}
\usepackage{lmodern}
\usepackage{geometry}
\usepackage{indentfirst}
\usepackage{titlesec}
\usepackage{fancyhdr}
\usepackage{appendix}

\newtheorem{theorem}{Theorem}[section]
\newtheorem{lemma}[theorem]{Lemma}
\newtheorem{definition}[theorem]{Definition}
\newtheorem{proposition}[theorem]{Proposition}
\newtheorem{corollary}[theorem]{Corollary}

\newtheoremstyle{remarkitalic} 
  {}{}                 
  {\normalfont}        
  {}                   
  {\itshape}           
  {.}                  
  { }                  
  {}                   

\theoremstyle{remarkitalic}
\newtheorem{remark}[theorem]{Remark}

\makeatother
\titleformat{\section}
  {\Large\bfseries}
  {\thesection}{0.5em}{}

\hypersetup{
    colorlinks = true,
    linkcolor  = red,
    urlcolor   = black,
    citecolor  = blue
}

\begin{document}
\title{Existence of multi-solitons for logarithmic Schrödinger
equation with a perturbation of a smooth $L^2$--subcritical function nonlinearity} 
	\author{Mohamed Bensaid\\Univ. Lille, CNRS, Inria, UMR 8524 - Laboratoire Paul Painlevé, F-59000 Lille, France\\ \url{mohamed.bensaid@univ-lille.fr}}
\maketitle
\begin{abstract}
    In this paper, we study the nonlinear Schrödinger equation with a nonlinearity combining a focusing logarithmic term and a smooth $L^2$--subcritical function.
    Our main result is the construction of multi-soliton solutions for this equation. These special configurations behave asymptotically, for large times, as the superposition of several solitons moving with distinct velocities. The proof relies on an iterative construction method combined with localized energy estimates adapted to the mixed structure of the nonlinearity. 
\end{abstract}
\section{Introduction }

\subsection{Settings }
We consider the nonlinear Schrödinger equation with a nonlinearity combining a logarithmic term and a smooth $L^2$--subcritical function.
\begin{equation}\label{NLS}\tag{NLSlog-g}
i\partial_t u + \Delta u + \lambda u \log|u|^2+ ug(|u|^2) = 0,
\qquad x \in \mathbb{R}^d,
\end{equation}
where $d \ge 1$ and $\lambda>0$.

This equation is of particular interest because it combines two types of nonlinearities that appear in different physical and mathematical contexts. 
Studying such a combination allows us to extend the known results for purely logarithmic or purely power-type nonlinearities.

In this work, we focus on the construction of multi-soliton solutions for \eqref{NLS} under suitable assumptions on the nonlinearity $g$, in \cite{ben2026}
the Cauchy problem has been shown to be globally well-posed in the Orlicz space $$W_1(\mathbb{R}^d)=\{u\in H^1(\mathbb{R}^d)\mid\abs{u}^2\log \abs{u}\in L^1(\mathbb{R}^d)\}.$$

The equation $\eqref{NLS}$ admits special traveling wave solutions called solitons: given a frequency $\omega$, an
initial phase $\gamma$, an initial position and speed $x,v \in \mathbb{R}^d$, and a solution $\Phi_{\omega} \in H^1(\mathbb{R}^d)$ of
\begin{equation}\label{050}
    \Delta\Phi_{\omega} - \omega\Phi_{\omega} + \lambda \Phi_{\omega} \log \abs{\Phi_{\omega}}^2 + \Phi_{\omega} g(\abs{\Phi_\omega}^2) = 0,
\end{equation}
a soliton solution of \eqref{NLS} traveling along the line $x = x_0 + v t$ is given by
\begin{equation}\label{000}
    Q_{\omega, \gamma, v, x_0}(t, x) := 
\Phi_\omega(x - v t - x_0)
e^{i\left( \frac{1}{2} v \cdot x - \frac{1}{4} |v|^2 t + \omega t + \gamma \right)}.
\end{equation}
The soliton resolution conjecture is a fundamental prediction for nonlinear dispersive equations. It states that any globally defined solution with finite energy asymptotically decomposes into a finite sum of isolated solitons, each moving at a constant velocity, together with a residual dispersive term. In other words, regardless of the initial complexity of the solution, the temporal evolution tends to isolate persistent stable structures (the solitons), while the excess energy is carried away by dispersive radiation.

The construction of multi-soliton solutions is directly motivated by this conjecture. A multi-soliton describes a solution whose long-time behavior is given by a superposition of solitons predicted by the conjecture. More precisely, let $N \geq 2$, 
$\omega_1, \dots, \omega_N \in \mathbb{R}$, 
$\gamma_1, \dots, \gamma_N \in \mathbb{R}$, 
$v_1, \dots, v_N \in \mathbb{R}^d$, 
$x_1, \dots, x_N \in \mathbb{R}^d$, 
and $\Phi_1, \dots, \Phi_N \in H^1(\mathbb{R}^d)$ solutions of the stationary equation~\eqref{050} 
(with $\omega$ replaced by $\omega_1, \dots, \omega_N$).  
Define for $j = 1, \dots, N$ the soliton components
\begin{equation}
    Q_j(t,x) := Q_{ \omega_j, \gamma_j, v_j, x_j}(t,x), 
\qquad 
Q(t,x) := \sum_{j=1}^{N} Q_j(t,x).
\end{equation}
Note that $Q_j$ satisfies
\begin{equation}\label{147}
    -\Delta Q_j
- Q_j \log|Q_j|^2
- Q_j g(|Q_j|^{2})
+\Big(\omega_j+\frac{|v_j|^2}{4}\Big) Q_j
+ i v_j \cdot \nabla Q_j
=0.
\end{equation}
And due to the nonlinearity, $Q(t,x)$ is generally not a solution of \eqref{NLS}. 
A \emph{multi-soliton} is a solution $u$ of \eqref{NLS}, defined on $[\Upsilon_0, +\infty)$ 
for some $\Upsilon_0 \in \mathbb{R}$, such that
$$
\lim_{t \to +\infty} \norm{ u(t) - Q(t) }_{H^1(\mathbb{R}^d)} = 0.
$$
\subsection{Nonlinear Schrödinger equation with power-type nonlinearity}
The idea begins with the classical case where $g$ is a pure power, i.e., 
$g(s) = \mu s^{\sigma}$, with $0 < \sigma < \frac{2}{d-2}$ if $d \geq 3$, 
and $0 < \sigma < \infty$ if $d = 1,2$, and $\lambda=0$. The equation reduces to the nonlinear power-type Schrödinger equation
\begin{equation}\label{NLSP}\tag{NLSP}
i\partial_t u + \Delta u + \mu |u|^{2\sigma}u = 0,
\qquad x \in \mathbb{R}^d.
\end{equation}
When $\mu < 0$, the Cauchy problem is globally well-posed in $H^1(\mathbb{R}^d)$ for $0<\sigma<\frac{2}{d-2}$. By contrast, when $\mu>0$, global well-posedness holds only under the condition
$0<\sigma<\frac{2}{d}$ (see \cite{1}). Moreover, the nonlinear Schrödinger equation possesses several invariances: invariance under spatial 
and temporal translations, Galilean invariance, and invariance under multiplication by a complex constant of modulus $1$. According to Noether's theorem, each invariance gives rise to a conservation law, namely energy, mass, and angular momentum.
$$\mathcal{E}(u)=\frac12\norm{\nabla u}_2^2-\frac{\mu}{2\sigma+2}\int\abs{u}^{2\sigma+2}\mathrm{d}x,\quad M(u)=\frac12\int \abs{u}^2\mathrm{d}x,\quad \mathcal{J}(u)=\frac12\Im\int \overline{u}\nabla u\mathrm{d}x.$$
The existence of soliton, which is unique, positive, radially symmetric, and radially decreasing, has been established in several works \cite{15,4,Kwong1989,Serrin,6,lecoz2009}, 
and we recall the following classical results:

\begin{itemize}
\renewcommand{\labelitemi}{$\bullet$}
\item For $0 < \sigma < \frac{2}{d}$, the solitons
are orbitally stable in $H^1(\mathbb{R}^d)$ (see \cite{Cazenave1982}). Multi-solitons were
constructed in this setting by Martel and Merle \cite{22}.

\item For $\sigma = \frac{2}{d}$, solitons are unstable, however, multi-solitons were
constructed by Merle \cite{Merle1990}.

\item For $\frac{2}{d} < \sigma < \frac{2}{d-2}$, solitons are unstable (see \cite{Gril}). Multi-solitons were
constructed in this setting by Martel, Merle, and Côte \cite{Cote2011}; see also Combet \cite{com}.
\end{itemize}
\subsection{Nonlinear Schrödinger equation with logarithmic-type nonlinearity}
Among nonlinear Schrödinger models, the logarithmic nonlinearity occupies a special role due to its distinctive analytical features, in particular its singular behavior at the origin and its slow growth at infinity. This leads to the logarithmic Schrödinger equation
\begin{equation}\label{NLSlog}\tag{logNLS}
i\partial_t u + \Delta u + \lambda u\log|u|^2 = 0,
\qquad x \in \mathbb{R}^d.
\end{equation}
It was introduced by Białynicki--Birula and Mycielski \cite{3}, to address a physical requirement
known as \emph{separability of independent systems}. 
$$
u_{0} = u_{1,0} \otimes u_{2,0}, \quad \text{i.e.} \quad
u_{0}(x_1,x_2) = u_{1,0}(x_1) u_{2,0}(x_2),
\quad \forall x_1 \in \mathbb{R}^{d_1}, \ \forall x_2 \in \mathbb{R}^{d_2},
$$
the solution $u$ to \eqref{NLSlog} in dimension $d=d_1+d_2$ with initial data $u|_{t=0}=u_{0}$ is
$$
u(t) = u_1(t) \otimes u_2(t),
$$
where $u_j$ is the solution to \eqref{NLSlog} in dimension $d_j$ with initial data $u_{j,0}$, for $j=1,2$.

Moreover, \eqref{NLSlog} possesses the same invariance as \eqref{NLSP} along with an additional scaling invariance: if $u = u(t,x)$ satisfies \eqref{NLSlog}, 
then for any $\kappa > 0$,
$$
u_\kappa(t,x) = \kappa  u(t,x)  e^{2 i \lambda t \log \kappa},
$$
also satisfies \eqref{NLSlog}. We also have a conservation laws, namely energy, mass, and angular momentum.
$$\mathcal{E}(u)=\frac12\norm{\nabla u}_2^2-\frac{\lambda}{2}\int\abs{u}^{2}(\log|u|^2-1)\mathrm{d}x,\quad M(u)=\frac12\int \abs{u}^2\mathrm{d}x,\quad J(u)=\frac12\Im\int \overline{u}\nabla u\mathrm{d}x.$$

The Cauchy theory for the logarithmic Schrödinger equation was first established by 
Cazenave and Haraux~\cite{8} in the case $\lambda>0$, in the Orlicz space $W_1(\mathbb{R}^d).$ 

The case $\lambda < 0$ was treated by Carles and Gallagher \cite{10}. The Cauchy theory was revisited and simplified by Hayashi and Ozawa \cite{19}.
The existence of a soliton has been established (see \cite{14}) for any $\omega \in \mathbb{R}$ and $\lambda > 0$. It is unique up to translation and is positive, radially symmetric, and radially decreasing. In this case, the solution is given explicitly and is called the \emph{Gausson}:
\[
\Phi_{\omega}(x) = e^{\frac{\omega + d}{2}} e^{-\frac{|x|^2}{2}}.
\]
We now recall the following result: The solitons
are orbitally stable in $W_1(\mathbb{R}^d)$ (see \cite{Ardila2016}). Multi-solitons were
constructed in this setting by G. Ferriere \cite{15}.
\subsection{Towards more general nonlinearities}
A first generalization of \eqref{NLSlog} is obtained by combining the logarithmic term with a power-type nonlinearity,
leading to the equation \begin{equation}\label{NLSlp}\tag{NLSlogP}
i\partial_t u + \Delta u + \lambda u\log\abs{u}^2+\mu u\abs{u}^{2\sigma} = 0,
\qquad x \in \mathbb{R}^d,\ d\ge 1.
\end{equation}

This equation was introduced by Carles and Gallagher \cite{13} in connection with a natural question, namely whether the large-time behavior of a power-type perturbation of equation \eqref{NLSlog} can be described.
Since the logarithmic nonlinearity has a strong influence on the dynamics (see \cite{13}),
the relevant comparison should be made with solutions of \eqref{NLSlog}, rather than with those
of \eqref{NLSP}. Due to the different methods used to study the cases of \eqref{NLSP} and \eqref{NLSlog}, 
an additional difficulty arises in the analysis of this equation.
The existence of solutions was established by Carles and Gallagher \cite{13} in the space $\Sigma_\alpha$
for $\lambda,\mu < 0$.
More recently, uniqueness in $H^1(\mathbb{R}^d)$ has been proved by Hayashi in~\cite{14}, without addressing existence. The existence of solitons has been established recently in \cite{32}, and the orbital stability of solitons is not yet established. In this work, we focus only on the construction of multi-solitons for \eqref{NLS}, which includes the case \eqref{NLSlp}. Recall that energy for \eqref{NLS} defined in the following way

\begin{equation}\label{0101}
E(u) := \frac{1}{2} \int |\nabla u|^2 \mathrm{d}x 
       - \frac{\lambda}{2} \int \abs{u}^2(\log\abs{u}^2-1)  \mathrm{d}x
       - \int G(|u|^2) \mathrm{d}x,
\end{equation}
where
$
G(y) := \frac12\int_0^{y} g(s) \mathrm{d}s
$ for all $y\in \mathbb{R_+}$. Inspired by the power-type NLS case, we assume that the function $g$ satisfies the following assumptions:
\begin{enumerate}[label=(A\arabic*), ref=(A\arabic*)]
\item \label{A1}
$g \in \mathcal{C}^0([0,+\infty),\mathbb{R})\cap \mathcal{C}^1((0,+\infty),\mathbb{R})$, 
$g(0)=0$, and
$
\lim_{s\to 0} s g'(s)=0.$
\item \label{A2}
There exist constants $C>0$ and $\sigma>0$ such that
$$
|g'(s)| \le C s^{\sigma-1}, \quad \text{for all } s\geq 1,
$$
where
$$
\begin{cases}
0<\sigma<\dfrac{2}{d-2}, & \text{if } d\ge 3,\\[6pt]
0<\sigma<+\infty, & \text{if } d=1,2,
\end{cases}
$$
\end{enumerate}
The assumptions \ref{A1}-\ref{A2} are introduced to ensure the local existence of solutions. In addition, we distinguish between the focusing and the defocusing cases:
\begin{enumerate}[label=(B\arabic*), ref=(B\arabic*)]
    \item \label{B1} Focusing case: There exists $s_0>0$ such that 
          $G(s_0)>0.$
    \item \label{B2} Defocusing case: For all $s_0>0$, 
          $G(s_0)\leq0.$
\end{enumerate}
\begin{remark}
When $\lambda = 0$ and $g$ satisfies assumptions \ref{A1}--\ref{A2}, the Cauchy problem is locally well-posed in $H^1$ (see \cite[Chapter 4]{6}). In addition, if assumption \ref{B1} also holds, the solution is globally well-posed provided that $\sigma < \frac{2}{d}$ (see \cite[Chapter 4]{6}). Moreover, solitons exist (see \cite{4}), and multi-solitons were constructed by Côte and Le Coz \cite{12}.
\end{remark}
We now turn to the case $\lambda > 0$, for which global well-posedness can be established under suitable assumptions on the nonlinearity.
\begin{theorem}\cite[Theorem 1.1]{ben2026}\label{100}
    For $\lambda > 0$ and \ref{B2}, 
or \ref{B1} with $0 < \sigma < \frac{2}{d}$, 
for any initial data $u_0 \in W_1(\mathbb{R}^d)$, there exists a unique global solution $
u \in \mathcal{C}_b(\mathbb{R}, W_1(\mathbb{R}^d))\cap\mathcal{C}^1(\mathbb{R},W_1'(\mathbb{R}^d)).
$ for \eqref{NLS},
Moreover, the mass $M(u(t))$, the angular momentum $\mathcal{J}(u(t))$, and the energy $E(u(t))$ are conserved in time.
\end{theorem}

\begin{definition}
    Let $c > 0$ be arbitrary. We say that $u_c \in S(c)$ is a \emph{ground state} if
$$
E(u_c) = \inf \big\{ E(u) \mid  u \in S(c) \big\},
$$
where $S(c):=\{u\in W_1(\mathbb{R}^d)\mid \norm{u}_2=c\}.$
\end{definition}
The existence of traveling wave solutions \eqref{050} is not guaranteed for all $\omega \in \mathbb{R}$, however, they do exist for certain values of $\omega$.
\begin{theorem}\cite{27}\label{06}
    Let $\lambda=1$. Suppose $d \ge 2$, $c > 0$, and that one of the following three conditions holds:
\begin{itemize}
    \item[$\bullet$] \ref{B2} holds and $\sigma > 0$;
    \item[$\bullet$] \ref{B1} holds and $0 < \sigma < \frac{2}{d}$;
\end{itemize}
Then the equation \eqref{050} admits a ground state, which is positive, radially symmetric, decreasing.
\end{theorem}
Note that the previous theorem can be extended to a wider class of nonlinearities. Under appropriate assumptions on the nonlinearity, together with a strong sublinearity condition, one can establish the existence of a ground state (see \cite{27}). However, uniqueness remains an open question in this setting, and the orbital stability of the ground state has not yet been established.

\subsection{Main result}

In this paper, we extend the techniques previously used to construct multi-solitons for the \eqref{NLSP} and \eqref{NLSlog} equations, in order to build multi-solitons for the \eqref{NLS}. This leads to the following main result.
\begin{theorem}\label{09}
 Let $N \geq 2$, 
$\omega_1, \dots, \omega_N \in \mathbb{R}$, 
$\gamma_1, \dots, \gamma_N \in \mathbb{R}$, 
$v_1, \dots, v_N \in \mathbb{R}^d$, 
$x_1, \dots, x_N \in \mathbb{R}^d$, and suppose that for each chosen $\omega_j$, the corresponding $\Phi_{\omega_j}$ exists. \\
We define
$$
v_\star := \min_{j \neq k} \lvert v_j - v_k \rvert .
$$
If $v_\star >0$, then there exists $\Upsilon_0 \in \mathbb{R}$ and a solution 
$
u \in \mathcal{C}([\Upsilon_0,\infty), W_1(\mathbb{R}^d))
$ of \eqref{NLS}
such that for all $ t \in [\Upsilon_0, \infty) $, 
$$
\norm{ u(t) - \sum_{j=1}^N Q_j(t)}_{H^1}
\leq 
\exp\!\left( -\frac{(v_\star (t-\Upsilon_0))^2}{8} \right),
$$where $Q_j$ is defined in~\eqref{000}.
\end{theorem}
\begin{remark}
    Note that the convergence is only obtained in $H^1$--norm here. However, we do believe that a convergence in $W_1(\mathbb{R}^d)$ should hold.
\end{remark}

For the nonlinear Schrödinger equation with a general nonlinearity of the form $g_1(|u|^2) u$, 
the function $g_1$ usually needs to satisfy a flatness property at $0$ for such a result to hold 
(for example, in \cite{22, 12}: $ g_1(0) = 0 $ and $\lim_{s \to 0} s g_1'(s) = 0 $). 
Indeed, setting $g_1(s) = \lambda \log(s) + g(s)$, we observe that $g_1(s)\to -\infty$ as $s\to 0^+$. A similar situation occurs in the case of \eqref{NLSlog}. 

This paper is organized as follows. In Section~\ref{1}, we prove that the ground state exhibits a Gaussian decay. Section~\ref{3} explains how the uniform $H^1$--estimates together with the compactness property yield the construction of multi-solitons. In Section~\ref{4}, we establish the main uniform $H^1$--estimates. Section~\ref{5} is devoted to proving the compactness property. The Appendix contains several technical lemmas along with their proofs. By the scaling
\[
u(t,x)\mapsto u\left(\lambda t,\sqrt{\lambda}\,x\right), \qquad \lambda>0,
\]
the equation is transformed into one with nonlinearity
\[
g_\lambda(s)=g(s)\lambda^{-1}.
\]
Since $g_\lambda$ satisfies the same assumptions \ref{A1}--\ref{A2} as $g$, it suffices to establish the result for $\lambda=1$. Accordingly, we shall assume $\lambda=1$ throughout.

\textbf{Notation:} We use the notation $A \lesssim B$ to denote the inequality $A \leq C B$ for some constant $C > 0$ which is independent of the time and of $n$. We also define $f^{+} := \max(0,f)$. For $f,g\in L^2(\mathbb{R}^d)$, we set  $$\Re(f,g):=\Re(\int f\overline{g})\quad \text{ and }\quad \Im(f,g):=\Im(\int f\overline{g}).$$



\section{ Properties of the solitary waves}\label{1}
Let $\Delta_r$ denote the radial Laplacian, that is, for any radial function $f=f(r)$,
$$
\Delta_r f = f''(r) + \frac{d-1}{r} f'(r).
$$

Consider the stationary equation \eqref{050}. From Theorem \ref{06}, we may assume that $\Phi_{\omega}$ is radial, positive, and strictly decreasing in $r=|x|$.  

In the case of \eqref{NLSP}, the ground state exhibits exponential decay (see \cite{4}). Moreover, in the case of \eqref{NLSlog}, the ground state can be computed explicitly and displays Gaussian decay (see \cite{5}). In our setting, for $r$ large and thus when the solution goes to 0, the logarithmic nonlinearity is stronger than a power nonlinearity, which suggests that one might expect a Gaussian-type decay. We obtain the precise decay properties in Lemma \ref{011} by means of the maximum principle. 

To justify the forthcoming computations rigorously, we need to ensure that the solution $\Phi_{\omega}$ of \eqref{050} is sufficiently regular.
\begin{lemma}
    Let $\Phi_{\omega}$ be the solution of \eqref{050}. Then $\Phi_{\omega}$ is of class $\mathcal{C}^2$.
\end{lemma}
\begin{proof} The idea is to use a technique similar to the one employed in the case of \eqref{NLSlog} (see \cite{14}). First, we remark that any non-negative solution of \eqref{050} satisfies
$$
-\Delta \Phi_{\omega} \leq q(\Phi_{\omega}), \quad \text{where} \quad q(\Phi_{\omega}) = -\omega \Phi_{\omega} + \Big(\Phi_{\omega} \log\Phi_{\omega}^2\Big)^+ -\Phi_{\omega}g(\abs{\Phi_\omega}^2).
$$
Since by assumption \ref{A1}-\ref{A2}, we have 
$$
|q(y)| \leq C_\sigma (1 + |y|^{2\sigma+1}) \quad \text{for all } y \in \mathbb{R},
$$
by repeating the argument of the proof of \cite[Lemma B.3]{33}, it is possible to show that 
$$
\Phi_{\omega} \in L^p_{\mathrm{\rm loc}}(\mathbb{R}^d) \quad \text{for all } p < \infty.
$$
Then, by standard regularity arguments applied to the equation \eqref{050}, we get $\Phi_{\omega}\in \mathcal{C}^2$.
\end{proof}

Now, we recall the standard Gaussian tail estimate.
\begin{lemma}\label{gausse}
    For any $m\in \mathbb{R}$ and any $r_0>0$, we have
$$
\int_{\abs{x}\geq r} \abs{x}^{m} e^{-\frac{(\abs{x}-r_0)^2}{2}}\mathrm{d}x
    \lesssim  r^{d+m-2} e^{-\frac{(r-r_0)^{2}}{2}},
    \qquad \forall r\ge \sqrt{2}+r_0.
$$
\end{lemma}
\begin{proof}
   By a polar change of variables, we have
\begin{align*}
    \int_{|x|\ge r} |x|^m e^{-(|x|-r_0)^2/2} \mathrm{d}x
= C_d \int_r^{\infty} \rho^{m+d-1} e^{-(\rho-r_0)^2/2} \mathrm{d}\rho&=C_d \int_{r-r_0}^{\infty} (\rho+r_0)^{m+d-1} e^{-\rho^2/2} \mathrm{d}\rho
\\&\leq C_d(r_0,|m|)  \Gamma\Big(\frac{m+d}{2}, \frac{(r-r_0)^2}{2}\Big),
\end{align*}
where, for $y \ge 1$ and $s \in \mathbb{R}$, $\Gamma$ is the incomplete Gamma function defined by
$$
\Gamma(s,y) = \int_y^\infty x^{s-1} e^{-x}  \mathrm{d}x.
$$
By induction, for all $\ell \in \mathbb{N}^*$, we have
$$
\Gamma(s,y) = \sum_{k=1}^{\ell} y^{s-k} e^{-y} +\Gamma(s-\ell, y) \prod_{k=1}^{\ell} (s-k)  .
$$
Now take $\ell = \lfloor s \rfloor + 1$ if $s-1>0$ and $\ell=1$ if $s-1\leq 0$ . Since $y \geq 1$, we deduce that
\begin{equation*}\label{gamma}
    \Gamma(s,y) \lesssim y^{s-1} e^{-y}.
\end{equation*}
Applying this inequality \eqref{gamma} with $y = \frac{(r-r_0)^2}{2}$ and $s = \frac{m+d}{2}$, we obtain
\begin{equation*}
    \int_{|x|\ge r} |x|^m e^{-(|x|-r_0)^2/2}  \mathrm{d}x
\lesssim r^{d+m-2} e^{-(r-r_0)^2/2}.\qedhere
\end{equation*}
\end{proof}
The following bounds on $\Phi_\omega$ are then established.
\begin{lemma}\label{011}
Let $\omega \in \mathbb{R}$ be such that $\Phi_\omega$ exists. Then there exist constants $C_1, C_2, r_0, \delta > 0$ such that, for every $r > 0$, the radial solution $\Phi_\omega$ of~\eqref{050} satisfies
$$
    C_1e^{-\delta r^2}\leq
    \Phi_{\omega}(r)
    \leq
    C_2 e^{-\frac{(r-r_0)^{2}}{2}} .
$$
\end{lemma}

\begin{proof}
We first prove the lower bound.  
Let $\delta>0$ and define
$$
   K_\delta(r) = e^{-\delta r^2}.
$$
A direct computation gives
\begin{align*}
    D(r):=\frac{-\Delta_r K_\delta + \omega K_\delta -  K_\delta\log(K_\delta^2) -  K_\delta g(\abs{K_\delta}^2)}{K_\delta}&=\Bigl(2d\delta+\omega\Bigr)
+\Bigl(2\delta-4\delta^2\Bigr) r^2
-
g(e^{-2\delta r^2})\\&\leq \Bigl(2d\delta+\omega\Bigr)
+\Bigl(2\delta-4\delta^2\Bigr) r^2
+C.
\end{align*}
There exist $\delta_1$ such that, for all $\delta> \delta_1$ and all $r\geq 1$, we have $$D(r)\leq 0.$$
Hence, for $r\geq 1$,
$$
-\Delta_r K_\delta + \omega K_\delta -  K_\delta\log(K_\delta^2) 
-  K_\delta g(\abs{K_\delta}^2) \le 0,
$$
while $\Phi_{\omega}$ satisfies
$$
-\Delta_r \Phi_{\omega} + \omega \Phi_{\omega} -  \Phi_{\omega}\log\Phi_{\omega}^2 -  \Phi_\omega g(\Phi_{\omega}^{2}) = 0.
$$
Define $v(r)=\Phi_{\omega}(r)-K_\delta(r)$. Then
$$
-\Delta_r v 
   +\omega v
   -\bigl(\Phi_{\omega}\log\Phi_{\omega}^2 - K_\delta\log(K_\delta^2)\bigr)
   - \bigl(\Phi_\omega g(\abs{\Phi_\omega}^2)-K_\delta g(\abs{K_\delta}^2)\bigr)
   \ge 0.
$$
Let
$$
H(s) = s \log s^2 +  sg(s^2).
$$
Since $H$ is continuous on $[0,+\infty)$ and of class $\mathcal{C}^1$ on $(0,+\infty)$, the mean value theorem implies that there exists
$$
\xi(r)\in (\min(\Phi_{\omega}(r),K_\delta(r)),\max(\Phi_{\omega}(r),K_\delta(r))),
$$
such that
$$
H(\Phi_{\omega}(r)) - H(K_\delta(r))
    = H'(\xi(r))(\Phi_{\omega}(r)-K_\delta(r)).
$$
Set
$$
q(r)=\omega - H'(\xi(r))=\omega-2-2\log(\xi(r))- (g(\xi(r)^2)+2\xi(r)^2g'(\xi(r)^2)).
$$
Then
$$
-\Delta_r v + q(r) v \ge 0.
$$
On the other hand, by assumption \ref{A1}--\ref{A2}, we have  $$
h(s):=\omega-2-2\log(s)- (g(s^2)+2 s^2g'(s^2))\geq\omega-2-2\log(s)-C (s^{2\sigma}+1),$$ then there exists $s_0>0$ such that, for all $0<s<s_0$,  
$$
h(s)>0,
$$
Consequently, for all
$
r>\sqrt{ \frac{1}{\delta_1}|\log(s_0)|},
$
we have
$$
K_\delta(r)\le K_{\delta_1}(r)<s_0.
$$
Since $\Phi_{\omega}(r)\to 0$ as $r\to +\infty$, there exists $r_1>0$ (independent of $\delta$) such that
$$
\Phi_{\omega}(r)<s_0 \qquad \text{for all } r\ge r_1.
$$
We then set
$$
r_0:=\max\left\{\sqrt{ \frac{1}{\delta_1}|\log(s_0)|},r_1\right\}+1.
$$
Therefore, $$
q(r)>0 \qquad \text{for all } r\ge r_0.
$$
We fixe $\delta > \delta_1$ such that $v(r_0) \ge 0$. Then, by the maximum principle, we deduce that
$$
\Phi_{\omega}(r)\ge K_\delta(r)=e^{-\delta r^2} \qquad  \text{ for all }r\ge r_0.
$$
By the decay of $r\mapsto\Phi_\omega(r)$, we deduce that for all $r \leq R_0$, we have $$\Phi_\omega(r)\geq\Phi_\omega(r_0)\geq \Phi_\omega(R_0)e^{\delta r_0^2}e^{-\delta r^2}.$$ 
Finally, for all $r> 0$, we get
$$
\Phi_{\omega}(r)\ge C_1e^{-\delta r^2} .
$$

The upper bound is proved directly using the elliptic equation
$$
-\Phi_{\omega}'' - \frac{d-1}{r}\Phi_{\omega}' + \omega \Phi_{\omega} - \Phi_{\omega} \log\Phi_{\omega}^2 - \Phi_{\omega} g(\abs{\Phi_{\omega}}^2) = 0.
$$
Using $\Phi_{\omega}'(r)<0$, we obtain
$$
-\Phi_{\omega}'' + \omega \Phi_{\omega} - \Phi_{\omega} \log\Phi_{\omega}^2 - \Phi_{\omega}g(\Phi_\omega^2) \leq 0.
$$
On the other hand, for sufficiently large $r$, we have $ g(\Phi_\omega^2)\leq 1$, hence
$$
-\Phi_{\omega}'' + \omega  \Phi_{\omega} - \Phi_{\omega} \log\Phi_{\omega}^2 \leq  \Phi_\omega g(\Phi_\omega^2)\leq \Phi_{\omega}.
$$
Multiplying this inequality by $\Phi_{\omega}' < 0$, we get
$$
-((\Phi_{\omega}')^2)' + (\omega - 1)(\Phi_{\omega}^2)' - (\Phi_{\omega}^2 \log\Phi_{\omega}^2 - \Phi_{\omega}^2)' \geq 0.
$$
Integrating this inequality over $[r, \infty)$, we deduce
$$
(\Phi_{\omega}')^2 - (\omega - 1) \Phi_{\omega}^2 + (\Phi_{\omega}^2 \log\Phi_{\omega}^2 - \Phi_{\omega}^2) \geq 0.
$$
Thus,
$$
(\Phi_{\omega}')^2 \geq \omega \Phi_{\omega}^2 - \Phi_{\omega}^2 \log\Phi_{\omega}^2.
$$
Taking the square root and dividing by $\Phi_{\omega} > 0$, we obtain, for $r\geq r_0$ such that $\omega-2\log\Phi_{\omega}\geq 0,$
$$
-\frac{\Phi_{\omega}'}{\Phi_{\omega}} \geq \sqrt{\omega - 2 \log(\Phi_{\omega})}.
$$
Hence,
$$
\frac{\rm d}{\mathrm{d}r} \left( \sqrt{\omega - 2 \log(\Phi_{\omega})} \right) \geq 1,
$$
which implies
$$
\sqrt{\omega - 2 \log(\Phi_{\omega}(r))} \geq r -r_0+\sqrt{\omega-2\log(\Phi_\omega(r_0))}\geq r-r_0,
$$
Consequently, for sufficiently large $r$,
$$
\Phi_{\omega}(r) \leq C' e^{-\frac{(r-r_0)^2}{2}}.
$$
Moreover, by continuity of the function $r \mapsto \Phi_{\omega}(r) e^{\frac{(r-r_0)^2}{2}}$, we finally deduce
\begin{equation*}
    \Phi_{\omega}(r) \leq C_2e^{-\frac{(r-r_0)^2}{2}},\quad \text{for all}~~ r>0.\qedhere
\end{equation*}
\end{proof}
\begin{corollary}\label{0111}
There exist $r_0>0$ such that for every $r>0$, the radial solution $\Phi_{\omega}$ of \eqref{050} satisfies
$$
    |\Phi_{\omega}'(r)| \lesssim r e^{-\frac{(r-r_0)^2}{2}} \quad \text{ and } \quad |\Phi_{\omega}''(r)| \lesssim (1+r^2) e^{-\frac{(r-r_0)^2}{2}}.
$$

\end{corollary}

\begin{proof}
Since $\Phi_{\omega}$ is continuous on $(0,+\infty)$, it suffices to prove the estimate for $r$ sufficiently large.
\\
For a radial function, we have
$$
    (r^{d-1}\Phi_{\omega}'(r))' = r^{d-1} f(\Phi_{\omega}(r)),
$$
where
$$
    f(\Phi_{\omega}) = -\Phi_{\omega} \log\Phi_{\omega}^{2} - \Phi_\omega g(\Phi_\omega^2).
$$
Using Lemma \ref{011}, we obtain
\begin{equation}\label{k}
    |g(\Phi_{\omega}(r))|
   \lesssim \Big(|\omega| + \Phi_{\omega}^{2\sigma}(r) + \abs{\log(\Phi_{\omega}^{2}(r))}\Big)\Phi_{\omega}(r)
   \lesssim (1+r^{2}) e^{-\frac{(r-r_0)^2}{2}}.
\end{equation}
Since the function $r \mapsto r^{d-1}g(\Phi_{\omega}(r))$ is integrable on $(0,\infty)$, then 
$
\ell := \lim_{r\to\infty} r^{d-1}\Phi_{\omega}'(r)>-\infty
$.
By the Gaussian decay, we must have $\ell=0.$
Consequently,
$$
|\Phi_{\omega}'(r)|
  \
  \lesssim
  \frac{1}{r^{d-1}} \int_{r}^{\infty} t^{d+1} e^{-\frac{(t-r_0)^2}{2}}\mathrm{d}t.
$$
Thanks to Lemma \ref{gausse} applied with $m=d+1$, we get
$$
|\Phi_{\omega}'(r)| \lesssim r e^{-\frac{(r-r_0)^2}{2}}.
$$
This completes the proof of the first estimate.\\
For the second estimate, we use \eqref{k} together with the bound obtained above and the equation
\begin{equation}
    \Phi_{\omega}''(r)
    = -\frac{d-1}{r}\Phi_{\omega}'(r) + f(\Phi_{\omega}(r)),
\end{equation}
which concludes the proof.
\end{proof}
\section{\texorpdfstring{Proof of the main result}{Proof of the main result}}\label{3}

Inspired by \cite{15,22,12}, the construction of a solution $u(t)$ satisfying the conclusion of Theorem \ref{09} is based on an approximate argument.\\
Let $(T_n)_{n\ge 1}$ be an increasing sequence of positive times satisfying $T_n \to +\infty$ as $n\to\infty$. \\
We consider $u_n$ the unique global solution $W_1(\mathbb{R}^d)$ given by Theorem \ref{100} (which approximates a
multi-soliton) of
\begin{equation}\label{030}
\begin{cases}
i\partial_t u_n + \Delta u_n + \lambda u_n \log|u_n|^2
+ u_ng(\abs{u_n}^2)  = 0,
\qquad (t,x)\in \mathbb{R}\times\mathbb{R}^d, \\[4pt]
u_n(T_n) = Q(T_n).
\end{cases}
\end{equation}

In \cite{22,12,15}, the proof of the construction relies on two main propositions: uniform estimates and a compactness argument.
The uniform backwards $H^1$--estimates 
for the solution are obtained by exploiting the slow variation 
of localized conservation laws and the $H^1$ coercivity 
of the action around $Q$, up to an $L^2$-norm term. These estimates are then established via a bootstrap argument. 
The compactness argument is obtained by applying a suitable cut-off technique.

\begin{proposition}[Uniform Estimates]\label{04}
  There exists $\Upsilon_0 > 0$ such that, for all $t \in [\Upsilon_0, T_n]$,
$$
\| u_n(t) - Q(t) \|_{H^1} \le \exp\!\Biggl( -\frac{(v_\star (t-\Upsilon_0))^2}{8} \Biggr).
$$

\end{proposition}
To prove Theorem \ref{09}, we also need to ensure that the initial data $u_n(\Upsilon_0)$ converge to some initial datum $u_0$, and that $u_n(t)$ converges weakly to $u(t)$ in $H^1(\mathbb{R}^d)$.
\begin{lemma}\cite[Theorem 1.3]{ben2026}\label{08}
   Let $T>0$ and $u_{n,0}, u_0 \in H^1(\mathbb{R}^d)$, and let $u_n(t)$ and $u(t)$ be the solutions corresponding to the initial data $u_{n,0}$ and $u_0$, respectively for \eqref{NLS}. Suppose that
$$
u_{n,0} \to u_0 \quad \text{in } L^2_{\rm loc}(\mathbb{R}^d).
$$ 
and that there exist $K_T>0$ such that $$\sup_{\substack{n \geq 0 \\ t \in [0,T]}}\norm{u_n(t)}_{H^1}\leq K_T.$$
Then, for all $t \in [0,T]$, we have 
$$
u_n(t) \rightharpoonup u(t) \quad \text{in } H^1(\mathbb{R}^d)\quad \text{ and }\quad u_n(t) \to u(t) \quad \text{in } L^2_{\rm loc}(\mathbb{R}^d).
$$
\end{lemma}
\begin{proposition}[Compactness]\label{07}
    There exist $u_0\in W_1(\mathbb{R}^d)$ such that (up to a subsequence) $$u_n(\Upsilon_0)\longrightarrow u_0 \quad \text{ in } L^2(\mathbb{R}^d).$$
\end{proposition}
Now the proof of Theorem~\ref{09} is an immediate consequence of 
these two propositions and Lemma \ref{08}, like in \cite{25,22,12}.

\begin{proof}[Proof of Theorem~\ref{09}]
By Proposition~\ref{07}, there exists $u_0 \in W_1(\mathbb{R}^d)$ such that, up to a subsequence,  
$$
u_n(\Upsilon_0) \longrightarrow u_0 \quad \text{in } L^2(\mathbb{R}^d).
$$
Since the Cauchy problem is well--posed in $W_1(\mathbb{R}^d)$, let $u \in \mathcal{C}_b([\Upsilon_0,\infty),W_1(\mathbb{R}^d))$ be the corresponding solution with initial data $u_0$ at $t=\Upsilon_0$, given by Theorem~\ref{100}.  

On the other hand, Proposition~\ref{04} ensures that $(u_n(t))$ is uniformly bounded in $H^1(\mathbb{R}^d)$, uniformly in $t$ and $n$.  
Therefore, by Lemma~\ref{08}, the sequence $(u_n(t))$ is converge weakly to $u(t)$ in $H^1(\mathbb{R}^d)$ for every $t \in [\Upsilon_0,\infty)$.
 Applying again Proposition~\ref{04}, we obtain
$$
\|u(t) - Q(t)\|_{H^1}
\leq
\liminf_{n\to\infty} \|u_n(t) - Q(t)\|_{H^1}
\leq
\exp\!\left( -\frac{(v_\star (t-\Upsilon_0))^2}{8} \right),
$$
for every $t\in [\Upsilon_0,\infty)$, which concludes the proof.
\end{proof}
\section{\texorpdfstring{Proof of the uniform $H^1$--estimates}{Proof of the uniform H1--estimates}}\label{4}
In \cite{15}, the uniform $L^2$--estimates for multi-solitons are obtained by direct computation. 
In \cite{25,22,12}, the uniform $H^1$--estimates for multi-solitons are established using a bootstrap argument. 
Here, we combine these two methods to prove the uniform estimates. 

The proof of Proposition \ref{04} is a consequence of the following bootstrap result.
\begin{lemma}[Bootstrap]\label{05}
There exist $\Upsilon_0,\Upsilon_1 > 0$ and $n_0 \in \mathbb{N}^*$ such that for all $n\geq n_0$, $\Upsilon_1<\Upsilon_0<T_n$ and, for all  $\Upsilon^\dagger \in [\Upsilon_0, T_n]$, the following holds: if for all $t \in [\Upsilon^\dagger, T_n]$ we have \begin{equation}\label{12}
        \norm{u_n(t) - Q(t)}_{H^1}
        \le e^{-\frac{(v_\star (t-\Upsilon_1))^2}{8}},
    \end{equation}
    then, for all $t\in [\Upsilon^\dagger, T_n]$ we have \begin{equation}\label{13}
        \|u_n(t) - Q(t)\|_{H^1}
        \le \frac12 e^{-\frac{(v_\star (t-\Upsilon_1))^2}{8}}.
    \end{equation}
\end{lemma}

\begin{proof}[Proof of Proposition \ref{04}]
    The idea is to consider the set
$$
\mathcal{T}_n := \Bigl\{ t \in [\Upsilon_0, T_n] \mid \|u_n(s) - Q(s)\|_{H^1} \le e^{-\frac{(v_\star (s-\Upsilon_1))^2}{8}} \quad \text{for all } s \in [t, T_n] \Bigr\}.
$$
Using the continuity of $u_n(t)$ in $H^1$ and the fact $u_n(T_n)=Q(T_n)$, one can show that the set $\mathcal{T}_n$ is both open and closed in $[\Upsilon_0,T_n]$ (equipped with the topology induced by $\mathbb{R}$) and not empty.  
By the connectedness of the interval $[\Upsilon_0,T_n]$, it follows that
$$
\mathcal{T}_n = [\Upsilon_0, T_n],
$$
which gives the desired uniform estimate.
\end{proof}  
The proof of Lemma~\ref{05} is carried out in two steps. 
First, assuming \eqref{12}, we show that the $L^2$-norm of $(u_n - Q)$ can be controlled. 
To obtain the full control of the $H^1$-norm of $(u_n - Q)$, as in \eqref{13}, we use the technique introduced in \cite{15}. 
The idea is to employ the weak linearization of an action-like functional, which we explain later. 
Indeed, in \cite{25,22,12}, the energy functional is of class $\mathcal{C}^2$, whereas in \cite{15} as well as for \eqref{NLS} this is not the case.
This weak linearization is coercive, i.e., it controls the $H^1$--norm up to a $L^2$--norm, 
all of which can be handled using the previously obtained $L^2$--estimate.

Let $\Upsilon_0,\Upsilon_1> 0$ to be fixed later and fix $n \in \mathbb{N}$ such that $T_n > \Upsilon_0>\Upsilon_1$. Set $w_n := u_n - Q$. Let $\Upsilon^\dagger \in [\Upsilon_0, T_n]$ and assume that for all $t \in [\Upsilon^\dagger, T_n]$, we have

\begin{equation}\label{as}
    \|w_n(t)\|_{H^1}
        \le e^{-\frac{(v_\star(t-\Upsilon_1))^2}{8}}.
\end{equation}
In the sequel, we set $\Upsilon_* := \Upsilon_0 - \Upsilon_1$. We choose $\Upsilon_0$ sufficiently large so that $\Upsilon_*$ is also sufficiently large. The parameter $\Upsilon_1$ will be fixed later.
\subsection{\texorpdfstring{Control of the $L^2$--norm}{Control of the L2-norm}}
In the case of the power-type NLS \cite{12}, the requires the velocity $v_\star$ to be sufficiently large. 
However, for \eqref{NLSlog}, this assumption is not needed. 
The difference comes from the decay properties of the solitary waves. In the former case, the decay is exponential, 
whereas in the logarithmic case it is faster than exponential, 
namely Gaussian (see \cite{15}). This compensates for the lack of a large-velocity assumption. This Gaussian decay is also found in the ground states of \eqref{NLS} (see Lemma \ref{011}) giving hope for a similar construction to work. 

To obtain the $L^2$--control in \cite{15}, the following  inequality is used:
\begin{lemma}\cite[Lemma 1.1.1]{8}\label{002}
    Let $z_1, z_2 \in \mathbb{C}$. Then the following estimate holds:
$$
\Big|\Im\Big( \bigl(z_1 \log|z_1| - z_2 \log|z_2| \big) \overline{(z_1 - z_2)} \Big)\Big|
   \le |z_1 - z_2|^{2}.
$$
\end{lemma}
While this estimate permits the control of the logarithmic part, in our situation it is also necessary to estimate the power term. To this end, we will need the following inequality:
\begin{lemma}\label{012}
   For all $t\geq 0$, we have $$\abs{\int(u_n(t)g(\abs{u_n(t)}^{2})-Q(t)g(\abs{Q(t)}^{2}))\overline{w_n(t)}}\lesssim \norm{w_n(t)}_{L^2(\mathbb{R}^d)}^2+\norm{w_n(t)}_{H^1}^{2\sigma+2}.$$
\end{lemma}
\begin{proof}
Applying Lemma~\ref{power} with $z_1 = u_n(t,x)$ and $z_2 = Q(t,x)$ and the fact $Q(t,x)$ is bounded in $L^\infty(\mathbb{R},L^\infty(\mathbb{R}^d))$ yields the desired conclusion.
\end{proof}
\begin{lemma}\label{bot1}
    For $t\in[\Upsilon^\dagger,T_n]$, we have $$
\norm{w_n(t)}_{L^2(\mathbb{R}^d)}
\lesssim
\Upsilon_*^{-\frac12} e^{-\frac{(v_\star(t-\Upsilon_1))^2}{8}}.
$$
\end{lemma}
\begin{proof}
First, $u_n(t)$ and $Q_j(t)$ are solutions of \eqref{NLS}, so that
$$
i\partial_t w_n + \Delta w_n + u_n \log|u_n|^2 - \sum_{j=1}^N Q_j \log|Q_j|^2 
+  u_n g(|u_n|^{2}) - \sum_{j=1}^N Q_j g(|Q_j|^{2})  = 0.
$$
Thus,
\begin{align*}
\frac{1}{2} \partial_t \|w_n(t)\|_{L^2}^2
&= \Im \langle i \partial_t w_n(t), w_n(t) \rangle_{H^{-1},H^1} \\
&= - \Im \bigg( u_n(t) \log|u_n(t)|^2 - \sum_{j=1}^N Q_j(t) \log|Q_j(t)|^2, w_n(t) \bigg) \\
&\quad - \Im \bigg( u_n(t) g(|u_n(t)|^{2}) - \sum_{j=1}^N Q_j(t) g(|Q_j(t)|^{2}), w_n(t) \bigg) \\
&= - \Im \bigg( u_n(t) \log|u_n(t)|^2 - Q(t) \log|Q(t)|^2, w_n(t) \bigg) \\
&\quad - \Im \bigg( Q(t) \log|Q(t)|^2 - \sum_{j=1}^N Q_j(t) \log|Q_j(t)|^2, w_n(t) \bigg) \\
&\quad - \Im \bigg( u_n(t) g(|u_n(t)|^{2}) - Q(t) g(|Q(t)|^{2}), w_n(t) \bigg) \\
&\quad - \Im \bigg( Q(t) g(|Q(t)|^{2}) - \sum_{j=1}^N Q_j(t) g(|Q_j(t)|^{2}), w_n(t) \bigg).
\end{align*}

We estimate these four terms for $t\in [\Upsilon^\dagger, T_n]$ using assumption \eqref{as}.
\\
The first term can be estimated using Lemma \ref{002}:
$$
\bigg| \Im \bigg( u_n(t) \log|u_n(t)|^2 - \sum_{j=1}^N Q_j(t) \log|Q_j(t)|^2, w_n(t) \bigg) \bigg| 
\le 2 \|w_n(t)\|_{L^2}^2 
\le 2 e^{-\frac{(v_\star (t-\Upsilon_1))^2}{4}}.
$$
The second term can be estimated using Lemma \ref{est6}:
\begin{align*}
\bigg| \Im \bigg( Q(t) \log|Q(t)|^2 &- \sum_{j=1}^N Q_j(t) \log|Q_j(t)|^2, w_n(t) \bigg) \bigg|
\\&\le \big\| Q(t)\log|Q(t)|^2 - \sum_{j=1}^N Q_j(t)\log|Q_j(t)|^2 \big\|_{L^2}\|w_n(t)\|_{L^2} \\
&\lesssim e^{-\frac{(v_\star (t-\Upsilon_1))^2}{4}}.
\end{align*}
The third term can be estimated using Lemma \ref{012}:
$$
\big| \Im( u_n(t) g(|u_n(t)|^{2}) - Q(t) g(|Q(t)|^{2}), w_n(t) ) \big| 
\lesssim \|w_n(t)\|_{L^2}^2 + \|w_n(t)\|_{H^1}^{2\sigma+2} 
\lesssim e^{-\frac{(v_\star (t-\Upsilon_1))^2}{4}}.
$$
The last term can be estimated using the Cauchy--Schwarz inequality and Lemma \ref{est5}:
\begin{align*}
    \big| \Im\big( Q(t) g(|Q(t)|^{2}) - \sum_{j=1}^N Q_j(t) g(|Q_j(t)|^{2}), w_n(t) \big) \big| 
&\le \big\| Q(t)g(|Q(t)|^{2}) - \sum_{j=1}^N Q_j(t)g(|Q_j(t)|^{2}) \big\|_{L^2}  \|w_n(t)\|_{L^2} 
\\&\lesssim e^{-\frac{(v_\star (t-\Upsilon_1))^2}{4}}.
\end{align*}
Gathering all these estimates, we obtain
$$
\frac{1}{2} \left| \partial_t \|w_n(t)\|_{L^2}^2 \right| \lesssim e^{-\frac{(v_\star (t-\Upsilon_1))^2}{4}}.
$$
Integrating this inequality over $(t,\infty)$ for $t\geq \Upsilon_0$ yields
\begin{equation*}
    \|w_n(t)\|_{L^2}^2 
\lesssim\int_t^\infty e^{-\frac{(v_\star (\tau-\Upsilon_1))^2}{4}}  \mathrm{d}\tau
\lesssim \Upsilon_*^{-1}  e^{-\frac{(v_\star (t-\Upsilon_1))^2}{4}}.\qedhere
\end{equation*}
\end{proof}
\subsection{\texorpdfstring{Control of the $H^1$--norm}{Control of the H1--norm}}

To obtain the control in $H^{1}$, we use an adaptation of techniques first developed by Martel, Merle, and Tsai
(see \cite{25}, and also \cite{12}). 
Indeed, it is known that each soliton $Q_j$ is a critical point,
of the action functional
$$
S_j(v) := E(v) + \left(\omega_j + \frac{|v_j|^2}{4}\right) M(v) + v_j \cdot \mathcal{J}(v).
$$
The proof of the bootstrap in $H^{1}(\mathbb{R}^d)$ relies on two key ingredients \emph{almost-coercivity} property 
and a \emph{slow-variation} property, similarly as in \cite{15}.  
In the power-type NLS case \cite{12}, one uses the linearization of $S_j$, and proves that the Hessian 
of these functionals is coercive on a finite-codimension of subspace of $H^{1}(\mathbb{R}^d)$ in $L^{2}(\mathbb{R}^d)$.

However, in the logarithmic case, the energy functional is not $\mathcal{C}^{2}$, so the classical second-order expansion cannot be used. As already mentioned, the key idea introduced in \cite{15} is to exploit 
a \emph{weak expansion} of the nonlinear term $z \mapsto \abs{z}^2 \log |z|$.  
This weaker form of linearization is nevertheless sufficient to recover the coercivity properties needed for the 
bootstrap, and the argument extends naturally to our setting.

For the slow-variation property, one needs to define quantities localized around each soliton.  
In the power-type NLS case, the localization must be essentially one-dimensional
because the decay of the solitons is only exponential, one must separate the solitons along a preferred direction 
to exploit their weak interaction see \cite{25,22,12}.
In contrast, in the logarithmic case the decay is \emph{Gaussian}, thus much faster. Thus, any localization along a single dimension in which all velocity components are well ordered 
would lead to a slower convergence rate, since the minimal relative speed in that direction would be smaller. 
The idea developed in \cite{15} is to "split" the space into regions centered around each soliton.

In our case, since we have proved that our soliton also decays according to a Gaussian profile 
(see Lemma~\ref{011}), we can use the same localization scheme as in the logarithmic framework 
treated in \cite{15}. 
This allows us to construct a partition of unity, so that each soliton can be treated separately. 
In this way, we can obtain local estimates around each soliton and then easily combine them 
to derive a global estimate for the full solution.

Let $\phi \in \mathcal{C}^{\infty}(\mathbb{R},\mathbb{R})$ be such that 
$$
\phi(s) = 
\begin{cases} 
1 & \text{for } s \le -1, \\
0 & \text{for } s \ge 1,
\end{cases}
\quad \text{with } \phi(s) \in [0,1] \text{ and } \phi'(s) \in [-1,0].
$$

We define:

\begin{itemize}
    \item The center of each soliton:
    $$
        x_j^*(t) := x_j + v_j t.
    $$
    
    \item The functions $\psi_j$ with the $j$-th member weighted around the $j$-th soliton:
    $$
        \psi_j(t,x) := \phi\Big(|x - x_j^*(t)| - \frac{v_\star (t-\Upsilon_1)}{2} - 2\Big).
    $$
    
    \item A last function $\psi_0$ completing the previous family so that $(\psi_j)_{1 \le j \le N}$ forms a partition of unity:
    $$
        \psi_0 := 1 - \sum_{j=1}^N \psi_j.
    $$
\end{itemize}

\begin{lemma}
    Assume that $$\Upsilon_1> \frac{6+\max_{j,k}\abs{x_j-x_k}}{v_\star}.$$ Then, for all $t\geq \Upsilon_1$, the family $(\psi_j(t))_{0\leq j\leq N}$ is a smooth partition of unity .
\end{lemma}
\begin{proof}
By definition of $\phi$ and $\psi_j$ for $j \geq 1$, we have, for $t \geq \Upsilon_1$,
 $$\operatorname{supp}\psi_j(t)\subset B\bigg(x_j^*(t),\frac{v_\star(t-\Upsilon_1)}{2}+3\bigg).$$
Thus, for $j\neq k$ we have  $$\operatorname{supp}\psi_j(t)\cap \operatorname{supp}\psi_k(t)\subset B\bigg(x_j^*(t),\frac{v_\star(t-\Upsilon_1)}{2}+3\bigg)\cap B\bigg(x_k^*(t),\frac{v_\star(t-\Upsilon_1)}{2}+3\bigg)=\emptyset.$$
    Indeed, \begin{equation*}
        \abs{x_j^*(t)-x_k^*(t)}\geq v_\star t-\max_{j,k}\abs{x_j-x_k}> v_\star (t-\Upsilon_1)+6.\qedhere
    \end{equation*}
\end{proof}
In the multi-soliton regime, the solitons interact weakly but in a nontrivial
manner. To capture this structure, it is essential to localize the action around each
soliton. 

We introduce localized versions of the energy, charge and momentum. For $j = 0, ..., N$ we define
$$M_j(t,v):=\frac{1}{2}\int\abs{v}^2\psi_j(t) \mathrm{d}x, \quad \mathcal{J}_j(t,v):=\frac12\Im\int \Bar{v}\nabla v \psi_j(t) \mathrm{d}x.$$
$$E_j(t,v):= \frac{1}{2}\int |\nabla v|^2\psi_j(t) \mathrm{d}x-\frac{1}{2}\int\abs{v}^2\big(\log\abs{v}^2-1\big)\psi_j(t) \mathrm{d}x-\int G(|v|^2)\psi_j(t) \mathrm{d}x.$$
\\
Then, we also define as well localized actions (for $j = 0, . . . , N $) by
$$S_j^{\text{\rm loc}}(v)=E_j(v)+(\omega_j+\frac{\abs{v_j}^2}{4})M_j(v)+v_j\cdot\mathcal{J}_j(v),$$
where for $j=0$ we take $\omega_0=0$ and $v_0=0$.
Finally, we define a localized action-like functional for multi-solitons:
$$S^{\text{\rm loc}}(v):=\sum_{j=1}^N S_j^{\text{\rm loc}}(v).$$

\subsubsection{Almost-coercivity}
\begin{proposition}[Almost-coercivity]\label{bot2}
    For all $t\in [\Upsilon^\dagger,T_n]$, we have
    $$
    S^{\rm loc}(t,u_n(t)) - \sum_{j=1}^N S_j(Q_j) 
    \geq \frac{1}{2} \|\nabla w_n(t)\|_{L^2}^2 - C \Upsilon_*^{-1/2} e^{-\frac{(v_\star (t-\Upsilon_1))^2}{4}}.
    $$
\end{proposition}
We first observe that
\begin{equation}\label{gui}
\begin{aligned}
S^{\rm loc}(t,u_n(t)) - \sum_{j=1}^N S_j(Q_j)
&= S_0(t,u_n(t))+\sum_{j=1}^N\Big( S^{\rm loc}_j(t,u_n(t)) -  S_j^{\rm loc}(t,Q_j(t)) \Big) \\
&\qquad- \sum_{j=1}^N \Big( S_j(Q_j) - S_j^{\rm loc}(t,Q_j(t)) \Big).
\end{aligned}
\end{equation}
We estimate each term separately. For the third term, we use several estimates on $Q_j$ outside its region of concentration. For the second term, the proof relies on linearizing the functional $S^{\rm loc}$ with respect to $w_n$ around the profile $Q$.

As already mentioned, the energy functional $E$ is not of class $\mathcal{C}^2$. Nevertheless, following the approach in \cite{15}, we introduce
$$
F(s) = |s|^2 \big(\log |s|^2 - 1\big).
$$
The following result, proved in \cite{15}, will play a key role in the sequel.
\begin{lemma}\cite[Lemma 3.18]{15}\label{pop}
    For all $z_1,z_2\in \mathbb{C}$, set $z=z_1-z_2$, then $$F(z_1)-F(z_2)-dF(z_2)(z)\leq 2 \abs{z}^2\bigl(\log(\max(\abs{z_1},\abs{z_2}))+1\bigr).$$
\end{lemma}
Taking $z_1=u_n(t,x)$ and $z_2=Q_j(t,x)$ we get 
\begin{corollary}\label{loglog}
    For all $x\in \mathbb{R}^d$, $t\geq 0$, $j=1,...,N$ and $n\in \mathbb{N}$, there holds 
    $$F(u_n(t,x))-F(Q_j(t,x))-dF(Q_j(t,x))(w_n^j(t,x))\leq 2 \abs{w_n^j(t,x)}^2\bigg(\log(\abs{w_n^j(t,x)}+1)+C\bigg).$$
\end{corollary}
\begin{proof}
    Taking $z_1=u_n(t,x)$ and $z_2=Q_j(t,x)$ in Lemma \ref{pop}, we get 
    $$F(u_n(t,x))-F(Q_j(t,x))-dF(Q_j(t,x))(w_n^j(t,x))\leq 2 \abs{w_n^j(t,x)}^2\bigl(\log(\max(\abs{u_n(t,x)},\abs{Q_j(t,x)}))+1\bigr).$$
    On the other hand, we have
$$\quad \norm{Q_j(t)}_{L^\infty}\leq C  \quad \text{ and } \abs{u_n(t,x)}\leq \abs{w_n^j(t,x)}+C.$$
Thus,
\begin{equation*}
    \log(\max(\abs{u_n(t,x)},\abs{Q_j(t,x)}))\leq \log(C+\abs{w_n^j(t,x)})\leq \abs{\log(C)}+\log(1+\abs{w_n^j(t,x)}).\qedhere
\end{equation*}

\end{proof}
\begin{lemma}\label{ma}
    For all $t\in [\Upsilon^\dagger,T_n]$ and for each $j = 1, \dots, N$, we have
    \begin{align*}
        S_j^{\text{\rm loc}}(t,u_n(t)) - S_j^{\text{\rm loc}}(t,Q_j(t))
    &\geq H_j(t,w_n^j(t)) 
    - \int \Re\!\big(\nabla Q_j(t) \overline{w_n^j(t)} \nabla \psi_j(t)\big)\mathrm{d}x
    \\&\quad 
    - \frac12v_j \cdot \Im\! \int \overline{Q_j(t)}  w_n^j(t) \nabla \psi_j(t)\mathrm{d}x,
    \end{align*}
    where $w_n^j = u_n - Q_j = w_n - \sum_{k \neq j} Q_k$ and
    \begin{align*}
        H_j(t,w) :&= \frac{1}{2} \int |\nabla w(t)|^2 \psi_j(t)\mathrm{d}x 
    - \int |w(t)|^2 \log\!\big(1 + |w(t)|\big) \psi_j(t)\mathrm{d}x 
    + \Bigl(\omega_j + \frac{|v_j|^2}{4}-C\Bigr) M_j(w) 
    \\&\quad + v_j \cdot \mathcal{J}_j(w) 
    - C \int |w(t)|^{2\sigma+2} \psi_j(t)\mathrm{d}x.
    \end{align*}
\end{lemma}

\begin{proof}
Recall that $
G(|z|^2)=\frac12\int^{\abs{z}^2}_0g(s)\mathrm{d}s=\int^{\abs{z}}_0 sg(s^2)\mathrm{d}s
$
, we have
$$
dF(Q_j)(w_n^j)
=2\Re\!\big(Q_j\overline{w_n^j}\big)\log(|Q_j|^2), \quad dG(|Q_j|^2)(w_n^j)
=g(|Q_j|^{2})\Re\big(Q_j\overline{w_n^j}\big).
$$
By Corollary \ref{loglog}, we have 
$$
-F(u_n)+F(Q_j)+dF(Q_j)(w_n^j)
\geq
-2|w_n^j|^2\big(\log(1+|w_n^j|)+C\big),
$$
and
\begin{align*}
     \abs{G(|u_n|^2)- G(|Q_j|^2)- dG(|Q_j|^2)(w_n^j)}
&=
\abs{\int_0^1
dG(|Q_j+\tau w_n^j|^2)(w_n^j)-dG(|Q_j|^2)(w_n^j)\mathrm{d}\tau}
\\& = \abs{\Re\int_0^1 \Big((Q_j+\tau w_n^j)g\big(\abs{Q_j+\tau w_n^j}^2\big)-Q_jg\big(\abs{Q_j}^2\big)\Big)\overline{w_n^j}\mathrm{d}\tau}
\\&\lesssim
|w_n^j|^2
+|w_n^j|^{2\sigma+2}.
\end{align*}
In the last inequality we use Lemma \ref{012} and the fact that $Q_j$ is bounded in $L^{\infty}$. Thus,
\begin{align*}
    E_j^{\rm loc}(u_n(t))-E_j^{\rm loc}(Q_j(t))&\geq \frac12\int\abs{\nabla w_n^j(t)}^2\psi_j(t)\mathrm{d}x-\int \abs{w_n^j(t)}^2\log(\abs{w_n^j(t)}+1)\psi_j(t)\mathrm{d}x\\&\quad-\int\Re(Q_j(t)\overline{w_n^j(t)})g(\abs{Q_j(t)}^2)\psi_j(t)\mathrm{d}x+\int \Re (\nabla Q_j(t)\overline{\nabla w_n^j(t))}\psi_j(t)\mathrm{d}x\\&\quad -\int \Re(Q_j(t)\overline{w_n^j(t)})\log(\abs{Q_j(t)}^2)\psi_j(t)\mathrm{d}x-C\int\abs{w_n^j(t)}^2\psi_j(t)\mathrm{d}x\\&\quad - C \int |w_n^j(t)|^{2\sigma+2} \psi_j(t)\mathrm{d}x.
\end{align*}
and 
\begin{align*}
    M_j^{\rm loc}(u_n(t))-M_j^{\rm loc}(Q_j(t))=\int\Re(Q_j(t)\overline{w_n^j(t)})\psi_j(t)\mathrm{d}x+\frac12\int\abs{w_n^j(t)}^2\psi_j(t)\mathrm{d}x.
\end{align*}

\begin{align*}
    \mathcal{J}_j^{\rm loc}(u_n(t))-\mathcal{J}_j^{\rm loc}(Q_j(t))&=\frac12\Im\int \overline{w_n^j(t)}\nabla w_n^j(t)\psi_j(t)\mathrm{d}x-\frac12\int \Im(\overline{Q_j(t)} w_n^j(t)\nabla\psi_j(t)\mathrm{d}x\\&\quad+\int \Im(\overline{w_n^j(t)}\nabla Q_j(t))\psi_j(t)\mathrm{d}x.
\end{align*}
Now observe that $Q_j$ satisfies
$$
-\Delta Q_j
- Q_j \log\!\big(|Q_j|^2\big)
- Q_j g(|Q_j|^{2})
+\Big(\omega_j+\frac{|v_j|^2}{4}\Big) Q_j
+ i v_j \cdot \nabla Q_j
=0.
$$
Multiplying this equation by $w_n^j(t)\psi_j(t)$, taking the real part, and integrating over $\mathbb{R}^d$, we deduce that
\begin{align*}
    &\int \Re (\nabla Q_j(t)\overline{\nabla w_n^j(t))}\psi_j(t)\mathrm{d}x-\int \Re(Q_j(t)\overline{w_n^j(t)})\log(\abs{Q_j(t)}^2)\psi_j(t)\mathrm{d}x
    \\&\quad-\int\Re(Q_j(t)\overline{w_n^j(t)})g(\abs{Q_j(t)}^{2})\psi_j(t)\mathrm{d}x+v_j\cdot\int \Im(\overline{w_n^j(t)}\nabla Q_j(t))\psi_j(t)\mathrm{d}x\\&\quad+\Big(\omega_j+\frac{|v_j|^2}{4}\Big)\int\Re(Q_j(t)\overline{w_n^j(t)})\psi_j(t)\mathrm{d}x\\&=-\int \Re (\nabla Q_j(t) \overline{w_n^j(t)})\nabla \psi_j(t)\mathrm{d}x.
\end{align*}
This concludes the proof of the lemma.
\end{proof}

We now prove the coercivity of $H_j$. The proof follows exactly the same approach as in \cite{15} for the logarithmic term. For the power term, we use the Sobolev embedding $H^1(\mathbb{R}^d) \hookrightarrow L^{2\sigma+2}(\mathbb{R}^d)$. For this we will need several estimates on $Q_j$ outside its "physical support". Thanks to Lemma~\ref{011}, we derive several estimates following the same computations as in \cite{15}.

For \( j \in \{1,\dots,N\} \), by Lemma~\ref{011} there exist constants 
\( C_1^{(j)}, C_2^{(j)}, r_0^{(j)}, \delta^{(j)} > 0 \) such that, for all 
\( (t,x) \in \mathbb{R}_+ \times \mathbb{R}^d \),
\[
C_1^{(j)} e^{-\delta^{(j)} |x - x_j^*(t)|^2}
\leq |Q_j(t,x)| \leq
C_2^{(j)} e^{-\frac{(|x - x_j^*(t)| - r_0^{(j)})^2}{2}}.
\]
Let
\begin{equation}\label{r}
    r_0 := \max_{1 \le j \le N} r_0^{(j)}.
\end{equation}
\begin{lemma}\label{est1}
Let $r_0$ be defined as in \eqref{r}, and assume that $$\Upsilon_1>\frac{6+2r_0+\max_{j,k}|x_j-x_k|}{v_\star}.$$Then for all $j,k \in \{1, \dots, N\}$ and for $0<\delta<\max(1,\frac{2}{d-2})$, there holds, for all $t\geq \Upsilon_1+\Upsilon_*$ and $\Upsilon_*$ large enough,
\[
\begin{aligned}
&\| Q_j(t) \|_{L^2\bigl((1-\psi_j(t))\mathrm{d}x\bigr)}
+ \| Q_j(t) \|_{L^2\bigl(| D_{xxx}^3 \psi_k(t) |\mathrm{d}x\bigr)}
+ \| Q_j(t) \|_{L^2\bigl(| \nabla \psi_k(t) |\mathrm{d}x\bigr)}
\lesssim (t-\Upsilon_1)^{-3} e^{-\frac{(v_\star (t-\Upsilon_1))^2}{8}}, \\
&\| \nabla Q_j(t) \|_{L^2\bigl((1-\psi_j(t))\mathrm{d}x\bigr)}
+ \| \nabla \psi_k(t) \nabla Q_j(t) \|_{L^2}
\lesssim \Upsilon_*^{-3} e^{-\frac{(v_\star (t-\Upsilon_1))^2}{8}}, \\
&\| Q_j(t) \|_{L^{2+\delta}\bigl((1-\psi_j(t))\mathrm{d}x\bigr)}
\lesssim \Upsilon_*^{-3} e^{-\frac{(v_\star (t-\Upsilon_1))^2}{8}}, \\
&\| Q_j(t) |x - x_j^*(t)|^2 \|_{L^2\bigl((1-\psi_j(t))\mathrm{d}x\bigr)}
+ \| Q_j(t) |x - x_j^*(t)|^3 \|_{L^2\bigl((1-\psi_j(t))\mathrm{d}x\bigr)}
\lesssim \Upsilon_*^{-3} e^{-\frac{(v_\star (t-\Upsilon_1))^2}{8}} .
\end{aligned}\]
\end{lemma}

For $k\neq j$ we have that $\psi_k\leq 1-\psi_j$, which allows us to conclude the following lemma.
\begin{lemma}\label{est2}
  Let $r_0$ be defined as in \eqref{r}, and assume that $$\Upsilon_1>\frac{6+2r_0+\max_{j,k}|x_j-x_k|}{v_\star}.$$ Then for all $k \ge 0$, $j \ge 1$ such that $k \neq j$ and for $0<\delta<\max(1,\frac{2}{d-2})$, there holds,  for all $t\geq \Upsilon_1+\Upsilon_*$ and $\Upsilon_*$ large enough,
\[
\begin{aligned}
&\| Q_j(t) \|_{L^2\bigl(\psi_k(t)\mathrm{d}x\bigr)}+\|  \psi_k(t)\nabla Q_j(t) \|_{L^2}
+ \| Q_j(t) \|_{L^{2+\delta}\bigl(\psi_k(t)\mathrm{d}x\bigr)}
\lesssim (t-\Upsilon_1)^{-3} e^{-\frac{(v_\star (t-\Upsilon_1))^2}{8}}, \\
& \| \nabla Q_j(t) \|_{L^2\bigl(\psi_k(t)\mathrm{d}x\bigr)}+
\| Q_j(t) |x - x_j^*(t)|^2 \|_{L^2\bigl(\psi_k(t)\mathrm{d}x\bigr)}
+ \| Q_j(t) |x - x_j^*(t)|^3 \|_{L^2\bigl(\psi_k(t)\mathrm{d}x\bigr)}
\lesssim \Upsilon_*^{-3} e^{-\frac{(v_\star (t-\Upsilon_1))^2}{8}} .\\
&\bigl\| Q_j(t) \log \lvert Q_j(t) \rvert^2 \bigr\|_{L^2\bigl((1-\psi_j(t))\mathrm{d}x\bigr)}+\bigl\| Q_j(t) \log \lvert Q_j(t) \rvert^2 \bigr\|_{L^2\bigl(\psi_k(t)\mathrm{d}x\bigr)}
\lesssim \Upsilon_*^{-3} e^{-\frac{ (v_\star (t-\Upsilon_1))^2}{8}} .
\end{aligned}
\]
\end{lemma}
In the following, we fix a time \begin{equation}\label{time}
    \Upsilon_1:=\frac{7+2r_0 + \max_{j,k} |x_j - x_k|}{v_\star}.
\end{equation}
\begin{lemma}[Coercivity]
For all $t\in [\Upsilon^\dagger,T_n]$ and for $j\in\{1,...,N\}$, we have
$$
H_j(t,w_n^j(t)) \geq \frac{1}{2} \int |\nabla w_n(t)|^2 \psi_j(t)  \mathrm{d}x - C \Upsilon_*^{-1/2} e^{-\frac{(v_\star (t-\Upsilon_1))^2}{4}} .
$$
\end{lemma}
\begin{proof} 
Recall that
\begin{align*}
H_j(t,w_n^j(t)) &:=
\frac{1}{2} \int |\nabla w_n^j(t)|^2 \psi_j(t)  \mathrm{d}x
- \int |w_n^j(t)|^2 \log\!\bigl(1+|w_n^j(t)|\bigr) \psi_j(t)  \mathrm{d}x \\
&\quad - C \int |w_n^j(t)|^{2\sigma+2} \psi_j(t)  \mathrm{d}x \\
&\quad + \Bigl(\omega_j + \frac{|v_j|^2}{4}-C\Bigr) M_j(t,w_n^j(t))
+ v_j \cdot \mathcal{J}_j(t,w_n^j(t)),
\end{align*}
and that $w_n^j = w_n + \sum_{k \neq j} Q_k$. Take $ \Upsilon_* $ large enough, then we can estimate each term using Lemma~\ref{est2}.

$\circ$ \textit{ First term (kinetic part).}
$$
\begin{aligned}
\int |\nabla w_n^j(t)|^2 \psi_j(t)  \mathrm{d}x
&= \int |\nabla w_n(t)|^2 \psi_j(t)  \mathrm{d}x
   + 2 \sum_{k \neq j} \operatorname{Re} \int \overline{\nabla w_n(t)}  \nabla Q_k(t)  \psi_j(t)  \mathrm{d}x \\
&\quad + \int \Bigl|\sum_{k \neq j} \nabla Q_k(t)\Bigr|^2 \psi_j(t)  \mathrm{d}x .
\end{aligned}
$$
For $k \neq j$, using the assumption \eqref{as} and Lemma \ref{est2} we obtain
$$
\Bigl|\operatorname{Re} \int \overline{\nabla w_n(t)}  \nabla Q_k(t)  \psi_j(t)  \mathrm{d}x\Bigr|
\le \|\nabla w_n(t)\|_{L^2} 
    \|\psi_j(t) \nabla Q_k(t)\|_{L^2}
\lesssim \Upsilon_*^{-1}e^{-\frac{(v_\star (t-\Upsilon_1))^2}{4}} .
$$
For the last term, again by Lemma \ref{est2},
$$
\Bigl\|\sum_{k \neq j} \nabla Q_k(t)\Bigr\|_{L^2(\psi_j(t)\mathrm{d}x)}
\le \sum_{k \neq j} \|\nabla Q_k(t)\|_{L^2(\psi_j(t)\mathrm{d}x)}
\lesssim \Upsilon_*^{-1/2} e^{-\frac{(v_\star (t-\Upsilon_1))^2}{8}} .
$$
Hence
$$
\int |\nabla w_n^j(t)|^2 \psi_j(t)  \mathrm{d}x
\ge \int |\nabla w_n(t)|^2 \psi_j(t)  \mathrm{d}x
   - C \Upsilon_*^{-1} e^{-\frac{(v_\star (t-\Upsilon_1))^2}{4}} .
$$

$\circ$\textit{ Second term (logarithmic nonlinearity):}
For all $\delta\in (0,1)$, there exists $C_\delta>0$ such that for all $y\ge0$,
$$
y^2 \log(1+y) \le C_\delta  y^{2+\delta},
$$
so
$$
\Bigl|\int |w_n^j(t)|^2 \log\!\bigl(1+|w_n^j(t)|\bigr) \psi_j(t)  \mathrm{d}x\Bigr|
\lesssim \|w_n^j(t)\|_{L^{2+\delta}(\psi_j(t)\mathrm{d}x)}^{2+\delta}.
$$
Choose $\delta>0$ such that the Sobolev embedding
$H^1(\mathbb{R}^d) \hookrightarrow L^{2+\delta}(\mathbb{R}^d)$ holds. 
Then, by Lemma \ref{est2},
$$
\begin{aligned}
\|w_n^j(t)\|_{L^{2+\delta}(\psi_j(t)\mathrm{d}x)}
&\le \|w_n(t)\|_{L^{2+\delta}} 
   + \sum_{k \neq j} \|Q_k(t)\|_{L^{2+\delta}(\psi_j(t)\mathrm{d}x)} \\
&\lesssim  \|w_n(t)\|_{H^1} +  \Upsilon_*^{-1} e^{-\frac{(v_\star (t-\Upsilon_1))^2}{8}} \\
&\lesssim  e^{-\frac{(v_\star (t-\Upsilon_1))^2}{8}} .
\end{aligned}
$$
Therefore
$$
\Bigl|\int |w_n^j(t)|^2 \log\!\bigl(1+|w_n^j(t)|\bigr) \psi_j(t)  \mathrm{d}x\Bigr|
\lesssim \Upsilon_*^{-1} e^{-\frac{(v_\star (t-\Upsilon_1))^2}{4}} .
$$

$\circ$\textit{ Third term (power nonlinearity).}
Using the Sobolev embedding $H^1(\mathbb{R}^d) \hookrightarrow L^{2\sigma+2}(\mathbb{R}^d)$,
$$
\|w_n^j(t)\|_{L^{2\sigma+2}(\psi_j(t)\mathrm{d}x)}
\le \|w_n(t)\|_{L^{2\sigma+2}} + \sum_{k \neq j} \|Q_k(t)\|_{L^{2\sigma+2}(\psi_j(t)\mathrm{d}x)}
\lesssim e^{-\frac{(v_\star (t-\Upsilon_1))^2}{8}}.
$$
Consequently,
$$
\Bigl|\int |w_n^j(t)|^{2\sigma+2} \psi_j(t)  \mathrm{d}x\Bigr|
\lesssim \Upsilon_*^{-1}e^{-\frac{(v_\star (t-\Upsilon_1))^2}{4}} .
$$

$\circ$\textit{ Last two terms (mass and momentum).} 
\begin{align*}
    M_j(t,w_n^j(t)) = \|w_n^j(t)\|_{L^2(\psi_j(t)\mathrm{d}x)}^2&\lesssim \norm{w_n(t)}_{L^2(\psi_j(t)\mathrm{d}x)}^2+\sum_{k\neq j}\norm{Q_k(t)}_{L^2(\psi_j(t)\mathrm{d}x)}^2\\&\lesssim \Upsilon_*^{-1} e^{-\frac{(v_\star (t-\Upsilon_1))^2}{4}},
\end{align*}
and
\begin{align*}
    |\mathcal{J}_j(t,w_n^j(t))| &\le \|w_n^j(t)\|_{L^2(\psi_j(t)\mathrm{d}x)} 
                         \|\nabla w_n^j(t)\|_{L^2(\psi_j(t)\mathrm{d}x)}\\
                         &\lesssim \Upsilon_*^{-1/2} e^{-\frac{(v_\star (t-\Upsilon_1))^2}{4}}.
\end{align*}
\textit{Conclusion.}
Gathering all the estimates, we obtain for $\Upsilon_*$ large enough
\begin{equation*}
    H_j(t,w_n^j(t)) \ge \frac12 \int |\nabla w_n(t)|^2 \psi_j(t)  \mathrm{d}x
                   - C \Upsilon_*^{-1/2} e^{-\frac{(v_\star (t-\Upsilon_1))^2}{4}} .\qedhere
\end{equation*}
\end{proof}

Finally to deduce the coercivity for the localized functionals $S_j^{\mathrm{\rm loc}}$ with $j \ge 1$, 
 we only need to show that the last two terms in Lemma \ref{ma} are negligible. For this we want the following lemma.

\begin{lemma}\label{est4}
   Let $\Upsilon_1$ be defined as in \eqref{time}. For all $j \neq k \in \{1, \dots, N\}$, $\ell \ge 1$, and $t\geq \Upsilon_1+\Upsilon_*$, there holds
$$
\begin{aligned}
\int \lvert Q_j(t) \rvert  \lvert Q_k(t) \rvert
\bigl( 1 + \lvert x - x_\ell^*(t) \rvert^2 \bigr)  \mathrm{d}x
+ \int \lvert \nabla Q_j(t) \rvert  \lvert Q_k(t) \rvert  \mathrm{d}x 
&+ \int \lvert \nabla Q_j(t) \rvert  \lvert \nabla Q_k(t) \rvert  \mathrm{d}x
\\&\lesssim \Upsilon_*^{-1} e^{-\frac{ (v_\star (t-\Upsilon_1))^2}{4}}.
\end{aligned}
$$
\end{lemma}
\begin{proof}
Note that for all  $j\in \{1,...,N\}$, we have $$\Big(|x - x_j^*(t)|-r_0^{(j)}\Big)^2=|x - x_j^*(t)|^2-2|x - x_j^*(t)|r_0^{(j)}+\big(r_0^{(j)}\big)^2\geq(|x - x_j^*(t)|-r_0)^2 -r_0^2+\big(r_0^{(j)}\big)^2.$$
We now estimate the different quantities appearing in the lemma using the Gaussian decay of $Q_j$ and $Q_k$.
$\circ$ \textit{The first term:} 
\begin{align*}
&\int |Q_j(t)|\, |Q_k(t)| \bigl( 1 + |x - x_\ell^*(t)|^2 \bigr)\, \mathrm{d}x \\
&\lesssim \int \bigl( 1 + |x - x_\ell^*(t)|^2 \bigr)
e^{-\frac{(|x - x_j^*(t)|-r_0)^2}{2}}
e^{-\frac{(|x - x_k^*(t)|-r_0)^2}{2}} \,\mathrm{d}x \\
&\lesssim \int \Bigg( 1 + \Big|y - x_\ell^*(t) + \frac{x_j^*(t) + x_k^*(t)}{2}\Big|^2 \Bigg) \\
&\quad \times
\exp\Bigg(
-\frac{1}{2}\Big(|y+\tfrac{x_j^*(t)-x_k^*(t)}{2}|-r_0\Big)^2
-\frac{1}{2}\Big(|y-\tfrac{x_j^*(t)-x_k^*(t)}{2}|-r_0\Big)^2
\Bigg)\,\mathrm{d}y
\end{align*}On the other hand \begin{align*}
    \Bigg(\abs{y+\frac{x_j^*(t)-x_k^*(t)}{2}}-r_0\Bigg)^2 &+\Bigg(\abs{y- \frac{x_j^*(t) - x_k^*(t)}{2}}-r_0\Bigg)^2\\&\geq 2\Bigg(|y|^2+\abs{\frac{x_j^*(t) - x_k^*(t)}{2}}^2-2r_0\bigg(|y|+\abs{\frac{x_j^*(t) - x_k^*(t)}{2}}\bigg)+r_0^2\Bigg)\\&=2\Bigg(\bigg(\abs{\frac{x_j^*(t) - x_k^*(t)}{2}}-r_0\bigg)^2+(|y|-r_0)^2-r_0^2\Bigg).
\end{align*}
Then, \begin{align*}
    \int &\lvert Q_j(t) \rvert  \lvert Q_k(t) \rvert
\bigl( 1 + \lvert x - x_\ell^*(t) \rvert^2 \bigr)  \mathrm{d}x\\&\lesssim \int \Bigg( 1 + \abs{y - x_\ell^*(t) + \frac{x_j^*(t) + x_k^*(t)}{2}}^2 \Bigg) 
\exp\Bigg(-\Big(|y|-r_0\Big)^2 -\Bigg(\abs{ \frac{x_j^*(t) - x_k^*(t)}{2}}-r_0\Bigg)^2\Bigg)  \mathrm{d}y
\\&\lesssim \exp\Bigg( -\Bigg(\frac{|x_j^*(t) - x_k^*(t)|}{2}-r_0\Bigg)^2 \Bigg)
\int \Bigl( |y|^2 + \Bigl| x_\ell^*(t) - \frac{x_j^*(t) + x_k^*(t)}{2} \Bigr|^2 + 1 \Bigr) e^{-(|y|-r_0)^2}  \mathrm{d}y\\
&\lesssim \Bigg( 1 + \abs{ x_\ell^*(t) - \frac{x_j^*(t) + x_k^*(t)}{2} }^2 \Bigg) 
\exp\Bigg( -\Bigg(\frac{|x_j^*(t) - x_k^*(t)|}{2}-r_0\Bigg)^2 \Bigg)\\
&\lesssim \Big( 1 + t^2 \Big) 
\exp\Bigg( -\Bigg(\frac{|x_j^*(t) - x_k^*(t)|}{2}-r_0\Bigg)^2 \Bigg)\\
&\lesssim \Upsilon_*^{-1} \exp\Big( -\frac{(v_\star (t - \Upsilon_1))^2}{4} \Big).
\end{align*}
In the last inequality, we use that
$$
|x_j^*(t) - x_k^*(t)|-2r_0
\geq v_\star t - |x_j - x_k|-r_0
= v_\star (t - \Upsilon_1) + v_\star \Upsilon_1 - |x_j - x_k|-2r_0
> v_\star (t - \Upsilon_1) + 6,
$$
and we take $\Upsilon_*$ sufficiently large.

$\circ$ \textit{The second term:} with a similar computation as the previous case,
\begin{align*}
     \int &\lvert \nabla Q_j(t) \rvert  \lvert Q_k(t) \rvert  \mathrm{d}x \\
     &\lesssim\int \abs{x-x_j^*(t)}e^{-\frac{(|x - x_j^*(t)|-r_0)^2}{2}}  e^{-\frac{(|x - x_k^*(t)|-r_0)^2}{2}}  \mathrm{d}x\\
     &\lesssim \int \left| y + \frac{x_k^*(t) - x_j^*(t)}{2} \right|\exp\Bigg(-\Big(|y|-r_0\Big)^2 -\Bigg(\abs{ \frac{x_j^*(t) - x_k^*(t)}{2}}-r_0\Bigg)^2\Bigg)  \mathrm{d}y\\
     &\lesssim \exp\Bigg( -\Bigg(\frac{|x_j^*(t) - x_k^*(t)|}{2}-r_0\Bigg)^2 \Bigg) \int (\abs{y}+|x_j^*(t)-x_k^*(t)|)e^{-(\abs{y}-r_0)^2}\mathrm{d}y\\
     &\lesssim \Big(1+\abs{x_j^*(t)-x_k^*(t)}\Big)\exp\Bigg( -\Bigg(\frac{|x_j^*(t) - x_k^*(t)|}{2}-r_0\Bigg)^2 \Bigg)\\
     &\lesssim \Upsilon_*^{-1}  \exp\Big( -\frac{(v_\star (t - \Upsilon_1))^2}{4} \Big).
\end{align*}

$\circ$ \textit{The last term:} with a similar computation as the first case,
\begin{align*}
&\int \left| \nabla Q_j(t) \right|  \left| \nabla Q_k(t) \right|  \mathrm{d}x 
\\&\lesssim \int |x - x_j^*(t)|  |x - x_k^*(t)|  
e^{-\frac{(|x - x_j^*(t)|-r_0)^2}{2}}  e^{-\frac{(|x - x_k^*(t)|-r_0)^2}{2}}  \mathrm{d}x \\
&\lesssim \int \left| y + \frac{x_k^*(t) - x_j^*(t)}{2} \right| 
\left| y - \frac{x_k^*(t) - x_j^*(t)}{2} \right| 
\exp\Bigg(-\Big(|y|-r_0\Big)^2 -\Bigg(\abs{ \frac{x_j^*(t) - x_k^*(t)}{2}}-r_0\Bigg)^2\Bigg)  \mathrm{d}y \\
&\lesssim \exp\Bigg( -\Bigg(\frac{|x_j^*(t) - x_k^*(t)|}{2}-r_0\Bigg)^2 \Bigg) 
\int \big( 1 + |y|^2 + |x_j^*(t) - x_k^*(t)|^2 \big)  e^{-(|y|-r_0)^2}  \mathrm{d}y \\
&\lesssim \exp\Big( -\Big(\frac{|x_j^*(t) - x_k^*(t)|}{2}-r_0\Big)^2 \Big) \big( 1 + |x_j^*(t) - x_k^*(t)|^2 \big) \\
&\lesssim \Upsilon_*^{-1}  \exp\Big( -\frac{(v_\star (t - \Upsilon_1))^2}{4} \Big).\qedhere
\end{align*}
\end{proof}
 \begin{lemma}
     For all $n \in \mathbb{N}$, $t \in [\Upsilon^\dagger, T_n]$, and $j \in \{1, \dots, N\}$, we have
     \begin{align*}
         &\biggl| \int \mathrm{Re}\bigl( \nabla Q_j(t)  w_n^j(t) \bigr)  \nabla \psi_j(t)  \mathrm{d}x \biggr|\lesssim \Upsilon_*^{-1} e^{-\frac{(v_\star (t-\Upsilon_1))^2}{4}},  \quad\int \abs{w_n^j(t)}^2\psi_j(t)\mathrm{d}x\lesssim \Upsilon_*^{-1}e^{\frac{(v_\star (t-\Upsilon_1))^2}{4}},\\&\qquad \hspace{3cm} \biggl| \mathrm{Im} \int Q_j(t)  w_n^j(t)  \nabla \psi_j(t)  \mathrm{d}x \biggr|  \lesssim   \Upsilon_*^{-1} e^{-\frac{(v_\star (t-\Upsilon_1))^2}{4}}.
     \end{align*}
 \end{lemma}
 \begin{proof}
     We observe that
$$
w_n^j = w_n + \sum_{k \ne j} Q_k,
$$
so that
$$
\begin{aligned}
\Bigl| \int \Re \bigl( \nabla Q_j(t)  w_n^j(t) \bigr) \cdot \nabla \psi_j(t)  \mathrm{d}x \Bigr|
&\le \Bigl| \int \Re \bigl( \nabla Q_j(t)  w_n(t) \bigr) \cdot \nabla \psi_j(t)  \mathrm{d}x \Bigr| \\
&\quad + \sum_{k \ne j} \Bigl| \int \Re \bigl( \nabla Q_j(t)  Q_k(t) \bigr) \cdot \nabla \psi_j(t)  \mathrm{d}x \Bigr|.
\end{aligned}
$$
By the assumption \eqref{as} and Lemma \ref{est1}, there holds, for all $t \in [\Upsilon^\dagger, T_n ]$ and take $\Upsilon_*$ large enough,
$$
\Bigl| \int \Re \bigl( \nabla Q_j(t)  w_n(t) \bigr) \cdot \nabla \psi_j(t)  \mathrm{d}x \Bigr|
\le \| w_n(t) \|_{L^2}  \| \nabla Q_j(t) \nabla \psi_j(t) \|_{L^2}
\lesssim \Upsilon_*^{-1} e^{-\frac{(v_\star (t-\Upsilon_1))^2}{4}  }.
$$
Moreover, Lemma \ref{est4} gives
$$
\Bigl| \int \Re \bigl( \nabla Q_j(t)  Q_k(t) \bigr) \cdot \nabla \psi_j(t)  \mathrm{d}x \Bigr|
\le \| \nabla Q_j(t) \|_{L^2}  \| Q_k(t) \nabla \psi_j(t)\|_{L^2}
\lesssim  \Upsilon_*^{-1} e^{-\frac{(v_\star (t-\Upsilon_1))^2}{4} },
$$
which gives the desired bound for the first term. A similar computation applies to the last one. 

As for the second one, for $\Upsilon_*$ large enough,
\begin{equation*}
    \int \abs{w_n^j(t)}^2\psi_j(t)\mathrm{d}x\leq \| w_n(t) \|_{L^2}^2 + \sum_{k \ne j} \| Q_k(t)\psi_j(t) \|_{L^2}^2
\lesssim \Upsilon_*^{-1} e^{-\frac{(v_\star (t-\Upsilon_1))^2}{4}}.\qedhere
\end{equation*}
 \end{proof}

\begin{corollary}\label{bot11}
    For all $t\in [\Upsilon^\dagger,T_n]$ and for each $j = 1, \dots, N$, we have $$S_j^{\text{\rm loc}}(t,u_n(t))-S_j^{\rm loc}(t,Q_j(t))\geq \frac12 \int |\nabla w_n(t)|^2 \psi_j(t)  \mathrm{d}x
                   - C  \Upsilon_*^{-1/2} e^{-\frac{(v_\star (t-\Upsilon_1))^2}{4}} .$$
\end{corollary}
To complete the proof, we also need the same property for $j = 0$. Using an approach similar to that in \cite{15}, we can prove the following lemma.

\begin{lemma}\label{bot22}
    For all $t \in [\Upsilon^\dagger, T_n]$,
    $$S_0^{\text{\rm loc}}(t, u_n(t)) \geq \frac{1}{2} \int |\nabla w_n(t)|^2 \psi_0(t)  \mathrm{d}x
                   - C  \Upsilon_*^{-1} e^{-\frac{(v_\star (t-\Upsilon_1))^2}{4}}.$$
\end{lemma}
\begin{proof}
Since $\omega_0=0$ and $v_0=0$, $S_0^{\rm loc}(t,u_n(t))=E_0(t,u_n(t))$ and we recall that its expansion is given by 
\begin{align*}
    S_0^{\mathrm{\rm loc}}(t, u_n(t)) 
&= \frac{1}{2} \int \lvert \nabla u_n(t) \rvert^2  \psi_0(t)  \mathrm{d}x
- \frac{1}{2} \int \lvert u_n(t) \rvert^2 \bigl( \log \lvert u_n(t) \rvert^2 - 1 \bigr)  \psi_0(t)  \mathrm{d}x \\&\quad - \int G(|u_n(t)|^2)\psi_0(t)\mathrm{d}x.
\end{align*}

$\circ$ \textit{The first term:}
$$
\int |\nabla u_n|^2  \psi_0(t)  \mathrm{d}x
= \int |\nabla w_n|^2  \psi_0(t)  \mathrm{d}x
+ 2 \int \mathrm{Re}  (\nabla w_n \cdot \nabla Q)  \psi_0(t)  \mathrm{d}x
+ \int |\nabla Q|^2  \psi_0(t)  \mathrm{d}x.
$$
The last two terms on the right-hand side can be easily estimated. 
First, we have
\begin{align*}
    \left| \int \mathrm{Re}  (\nabla w_n \cdot \nabla Q)  \psi_0(t)  \mathrm{d}x \right|
\le \|\nabla w_n\|_{L^2}  \|\psi_0(t) \nabla Q(t)\|_{L^2} 
&\le \|\nabla w_n\|_{L^2} \sum_{j \ge 1} \|\psi_0(t) \nabla Q_j(t)\|_{L^2} 
\\&\lesssim  \Upsilon_*^{-1} e^{-\frac{(v_\star (t-\Upsilon_1))^2}{4}}.
\end{align*}
Then we use Lemma \ref{est2} to estimate the last term.

$\circ$ \textit{The second term:} For all $\delta > 0$ and all $y > 1$, we have
$$
\log(y) \le \frac{1}{\delta}  y^{\delta}.
$$
 Thus,
 \begin{align*}
     \int |u_n|^2 \bigl(\log |u_n|^2 - 1\bigr)  \psi_0(t)  \mathrm{d}x&=\int |u_n|^2 \log (\frac{|u_n|^2}{e}) \psi_0(t)  \mathrm{d}x
\\&\le \int_{|u_n|^2 > e} |u_n|^2 \log (\frac{|u_n|^2}{e}) \psi_0(t)  \mathrm{d}x
\lesssim  \| u_n(t) \|_{L^{2 + \delta}(\psi_0(t)\mathrm{d}x)}^{2 + \delta}.
 \end{align*}
Then we choose $0<\delta<1$ such that the Sobolev embedding
$H^1(\mathbb{R}^d)\hookrightarrow L^{2+\delta}(\mathbb{R}^d)$ holds.
Then
$$
\begin{aligned}
\| u_n(t) \|_{L^{2 + \delta}(\psi_0(t) \mathrm{d}x)}
\le \| w_n(t) \|_{L^{2 + \delta}} 
+ \sum_{j \ge 1} \| Q_j(t) \|_{L^{2 + \delta}(\psi_0(t) \mathrm{d}x)} &\lesssim \| w_n(t) \|_{H^1}
+ \Upsilon_*^{-1}e^{-\frac{(v_\star (t-\Upsilon_1))^2}{8}} \\&\lesssim e^{-\frac{(v_\star (t-\Upsilon_1))^2}{8}} .
\end{aligned}
$$
In the second inequality, we have used Lemma~\ref{est2}.

$\circ$ \textit{The third term:} Thanks to Lemma \ref{power} and the Sobolev embedding $H^1(\mathbb{R}^d)\hookrightarrow L^{2\sigma+2}(\mathbb{R}^d)$, and using the same computation as for the second term, we obtain
\begin{align*}
\int |G(|u_n(t)|^2)|\psi_0(t) \mathrm{d}x
&\lesssim \norm{u_n(t)}_{L^2(\psi_0(t)dx)}^2+\| u_n(t) \|_{L^{2\sigma+2}(\psi_0(t) \mathrm{d}x)}^{2\sigma+2} \lesssim \Upsilon_*^{-1}e^{-\frac{(v_\star (t-\Upsilon_1))^2}{4}} .
\end{align*}
Indeed $$\norm{u_n(t)}_{L^2(\psi_0(t)dx)}\lesssim \norm{w_n(t)}_{L^2}+\sum_{j\geq 1}\norm{Q_j(t)}_{L^{2}(\psi_0(t)\mathrm{d}x)}\lesssim \Upsilon_*^{-1/2} e^{-\frac{(v_\star (t-\Upsilon_1))^2}{8}},$$ and $$\norm{u_n(t)}_{L^{2\sigma+2}(\psi_0(t)dx)}\lesssim \norm{w_n(t)}_{L^{2\sigma+2}}+\sum_{j\geq 1}\norm{Q_j(t)}_{L^{2\sigma+2}(\psi_0(t)\mathrm{d}x)}\lesssim  e^{-\frac{(v_\star (t-\Upsilon_1))^2}{8}}.$$
\end{proof}
Now return to estimate the third term as in \eqref{gui}
 \begin{lemma}\label{bot33}
   Let $\Upsilon_1$ be defined as in \eqref{time}. For all $j\in \{1,...,N\}$, for all $t\geq \Upsilon_1+\Upsilon_*$ and $\Upsilon_*$ large enough, $$\abs{ S_j(Q_j) - S_j^{\rm loc}(t,Q_j(t))}\lesssim \Upsilon_*^{-1/2}e^{-\frac{ (v_\star (t-\Upsilon_1))^2}{8}}.$$
\end{lemma}
\begin{proof}
With a simple computation,
    \begin{align*}
        S_j(Q_j) - S_j^{\rm loc}(t,Q_j(t))&=\frac12\int\abs{Q_j(t)}^2(1-\psi_j(t))\mathrm{d}x-\frac{1}{2}\int\abs{Q_j(t)}^2\big(\log\abs{Q_j(t)}^2-1\big)(1-\psi_j(t))\mathrm{d}x\\&\qquad-\int G(|Q_j(t)|^2)(1-\psi_j(t))\mathrm{d}x+\bigg(\omega_j+\frac{\abs{v_j}^2}{4}\bigg)\int\abs{Q_j(t)}^2(1-\psi_j(t))\mathrm{d}x\\&\qquad+v_j\cdot\Im\int\overline{Q_j(t)}\nabla Q_j(t)(1-\psi_j(t))\mathrm{d}x.
    \end{align*}
   Combining Lemma \ref{est1} and Lemma \ref{est2} with the last estimate in Lemma \ref{power}, we obtain the desired conclusion. 
\end{proof}
\begin{proof}[Proof of Proposition \ref{bot2}]
   We conclude the proof using the decomposition in \eqref{gui}. Each term is then estimated by applying Corollary \ref{bot11}, Lemma \ref{bot22}, and Lemma \ref{bot33}.
\end{proof}
\subsubsection{Slow variational}
Since the functional $S^{\mathrm{\rm loc}}$ is defined via a time-dependent partition of unity, it is not conserved along the flow, unlike the total energy. Therefore, it is necessary to control its time evolution in order to compare $S^{\mathrm{\rm loc}}(t,u_n(t))$ with the sum of the individual actions $S(Q_j)$. 

The following proposition ensures that this variation remains uniformly small over the entire time interval under consideration, provided that the solution stays close to the multi-soliton profile. This control is a crucial ingredient for propagating the coercivity properties established previously, and for preventing any significant transfer of mass or energy through the transition regions, thereby allowing the bootstrap argument to be closed.
\begin{proposition}(Slow variations)\label{slow}
    For all $t\in [\Upsilon^\dagger,T_n]$, $$S^{\text{\rm loc}}(t,u_n(t))-\sum_{j\geq 1}S_j(Q_j)\lesssim \Upsilon_*^{-1}e^{-\frac{(v_\star (t-\Upsilon_1))^2}{4}}.$$
\end{proposition}
The proof of this proposition is essentially the same as in \cite{25,22,15,12}. First, we decompose the left-hand side:  
\begin{equation}\label{10}
\begin{aligned}
S^{\mathrm{\rm loc}}(t, u_n(t)) - \sum_{j=0}^N S(Q_j)
&=
\bigl( S^{\mathrm{\rm loc}}(t, u_n(t)) - S^{\mathrm{\rm loc}}(T_n, Q(T_n)) \bigr)
+ S^{\mathrm{\rm loc}}_0(T_n, Q(T_n))\\
&\quad
+ \sum_{j \ge 1} \bigl( S^{\mathrm{\rm loc}}_j(T_n, Q(T_n)) - S^{\mathrm{\rm loc}}_j(T_n, Q_j(T_n)) \bigr)\\
&\quad
+ \sum_{j \ge 1} \bigl( S^{\mathrm{\rm loc}}_j(T_n, Q_j(T_n)) - S_j(Q_j) \bigr).
\end{aligned}
\end{equation}
Then, we estimate each term in the last equality. For the first term, 
remark that $u_n(T_n) = Q(T_n)$. Define
$$
\tilde{S}_n(t) := S^{\text{\rm loc}}\bigl(t, u_n(t)\bigr),
$$
so that
$$
\abs{S^{\mathrm{\rm loc}}(t, u_n(t)) - S^{\mathrm{\rm loc}}(T_n, Q(T_n))}
= \abs{\int_t^{T_n} \frac{d\tilde{S}_n(s)}{\mathrm{d}s}  \mathrm{d}s}.
$$
To estimate the derivative of $\tilde{S}_n(t)$ we first need to establish a link between the time and spatial derivatives of the partition.
\begin{lemma}\cite[Lemma 3.22]{15}\label{der}
$\forall j \ge 1, \ \forall t \ge 0, \ \forall x \in \mathbb{R}^d,$
$$|\partial_t \psi_j(t,x)| \lesssim  |\nabla \psi_j(t,x)|.$$
\end{lemma}
Similarly as in \cite{15}, some modifications are required due to the power-type perturbation; however, these can be handled using the Sobolev embedding and assumption \eqref{as}, yielding.
\begin{lemma}
For all $ t \in [\Upsilon^\dagger, T_n ],$ 
$$
\left| \frac{\rm d}{\mathrm{d}t} \tilde{S}_n(t) \right| \lesssim  e^{ - \frac{ (v_\star (t-\Upsilon_1))^2}{8} }.$$
\end{lemma}
\begin{proof}
First, remark that \begin{align*}
    \tilde{S}_n(t)&=\sum_{j=0}^N \bigl(E_j(u_n(t))+(\omega_j+\frac{\abs{v_j}^2}{4})M_j(t,u_n(t))+v_j\cdot\mathcal{J}_j(t,u_n(t))\bigr)\\&=E(u_n(t))+\sum_{j=1}^N (\omega_j+\frac{\abs{v_j}^2}{4})M_j(t,u_n(t))+\sum_{j=1}^N v_j\cdot\mathcal{J}_j(t,u_n(t)).
\end{align*}
Since the energy $E(u(t))$ is conserved along the flow of \eqref{NLS}, to estimate the variations of $S^{\text{\rm loc}}(t, u(t))$ it suffices to study the variations of the localized masses $M_j(t, u(t))$ and momenta $\mathcal{J}_j(t, u(t))$.  \\
Take any $j = 1, \dots, N$. We have
$$
\frac{\rm d}{\mathrm{d}t} M_j(t,u_n(t))
=
\int \mathrm{Im}(\Delta u_n  \overline{u_n})  \psi_j  \mathrm{d}x
+ \int |u_n|^2  \partial_t \psi_j  \mathrm{d}x
=
\int \mathrm{Im}(\nabla u_n  \overline{u_n}) \cdot \nabla \psi_j  \mathrm{d}x
+ \int |u_n|^2  \partial_t \psi_j  \mathrm{d}x.
$$
$$
\left| \frac{\rm d}{\mathrm{d}t} M_j(t,u_n(t)) \right|
\lesssim \int \left( |\nabla u_n|^2 + |u_n|^2 \right) |\nabla \psi_j|  \mathrm{d}x.
$$
For $\mathcal{J}_j(t, u_n(t))$, we compute the derivative of the $k$-th component. Let $k\in \{1,...,d\}$. Then
\begin{align*}
\frac{\rm d}{\mathrm{d}t} \int \mathrm{Im} \big( \partial_k u_n(t)  \overline{u_n(t)} \big)  \psi_j(t)  \mathrm{d}x
&= -\int \Re \big(i\partial_t \partial_k u_n(t) \overline{u_n(t)}\big)\psi_j(t)\mathrm{d}x+\int \Re\big(\overline{\partial_k u_n(t)} i\partial_t u_n(t)\big)\psi_j(t)\mathrm{d}x\\&\quad+\int \mathrm{Im} \big( \partial_k u_n(t)  \overline{u_n(t)} \big)  \partial_t\psi_j(t)  \mathrm{d}x.
\end{align*}
On the other hand, one has \begin{align*}
    \int \Re\Big( -\Delta u_n(t)\overline{\partial_k u_n(t)}\psi_j(t)+\partial_k \Delta u_n(t) \overline{u_n(t)}\psi_j(t)\Big)\mathrm{d}x&=\int \Re\big(\nabla u_n(t)\cdot \nabla \psi_j(t)\overline{\partial_k u_n(t)}\big)\mathrm{d}x\\&\quad-\int \Re\big(\partial_k\nabla u_n(t)\cdot\overline{u_n(t)}\nabla \psi_j(t)\big)\mathrm{d}x\\&=2\int \Re\big(\nabla u_n(t)\cdot \nabla \psi_j(t)\overline{\partial_k u_n(t)}\big)\mathrm{d}x\\&\quad+\int \Re\big(\partial_k u_n(t) \overline{u_n(t)}\Delta\psi_j(t)\big)\mathrm{d}x.
\end{align*}
An integration by parts yields
$$\int \Re\big(\partial_k u_n(t) \overline{u_n(t)}\Delta\psi_j(t)\big)\mathrm{d}x=-\frac12\int \abs{u_n(t)}^2 \partial_k \Delta\psi_j(t)\mathrm{d}x.$$
Then
\begin{align*}
    \int \Re\Big( -\Delta u_n(t)\overline{\partial_k u_n(t)}\psi_j(t)+\partial_k \Delta u_n(t) \overline{u_n(t)}\psi_j(t)\Big)\mathrm{d}x&=2\int \Re\big(\nabla u_n(t)\cdot \nabla \psi_j(t)\overline{\partial_k u_n(t)}\big)\mathrm{d}x\\&\quad -\frac12\int \abs{u_n(t)}^2 \partial_k \Delta\psi_j(t)\mathrm{d}x.
\end{align*}
For the logarithmic part, one obtains \begin{align*}
    \int \Re\Big(-\overline{u_n(t)}\partial_k (u_n(t)\log\abs{u_n(t)}^2)+u_n(t)\log\abs{u_n(t)}^2\overline{\partial_k u_n(t)}\Big)\psi_j(t)\mathrm{d}x&=-\int \partial_k \abs{u_n(t)}^2 \psi_j(t)\mathrm{d}x\\&=\int  \abs{u_n(t)}^2 \partial_k\psi_j(t)\mathrm{d}x.
\end{align*}
For the power term, we have \begin{align*}
    \int \Re\Big(-\overline{u_n(t)}&\partial_k (u_n(t)g(\abs{u_n(t)}^{2})+u_n(t)g(\abs{u_n(t)}^{2})\overline{\partial_k u_n(t)}\Big)\psi_j(t)\mathrm{d}x\\&=-\int \abs{u_n(t)}^2\partial_k \Big(g(\abs{u_n(t)}^{2})\Big) \psi_j(t)\mathrm{d}x\\&=-2\int G(|u_n(t)|^2)\partial_k \psi_j(t)+\int|u_n(t)|^2g(\abs{u_n(t)}^2)\partial_k\psi_j(t).
\end{align*}
In the second equality, we used 
\begin{align*}
    2\int \psi_j(t)\partial_k G(|u_n(t)|^2) &=\int g(\abs{u_n(t)}^2)\partial_k \abs{u_n(t)}^2 \psi_j(t)\\&=-\int g(\abs{u_n(t)}^2)\abs{u_n(t)}^2\partial_k\psi_j(t)-\int \abs{u_n(t)}^2\partial_k\Big(g(\abs{u_n(t)}^2)\Big) \psi_j(t)
\end{align*}$$ $$
Finally, combining all these calculations, we obtain \begin{align*}
    \frac{\rm d}{\mathrm{d}t} \int \mathrm{Im} \big( \nabla u_n(t)  \overline{u_n(t)} \big)  \psi_j(t)  \mathrm{d}x &=2\int \Re\big(\nabla u_n(t)\cdot \nabla \psi_j(t)\overline{\nabla u_n(t)}\big)\mathrm{d}x-\frac12\int \abs{u_n(t)}^2 \nabla \Delta\psi_j(t)\mathrm{d}x\\&\qquad +\int  \abs{u_n(t)}^2 \nabla\psi_j(t)\mathrm{d}x-2\int G(|u_n(t)|^2)\nabla \psi_j(t)\mathrm{d}x\\&\qquad +\int|u_n(t)|^2g(\abs{u_n(t)}^2)\nabla\psi_j(t)\mathrm{d}x+\int \mathrm{Im} \big( \nabla u_n(t)  \overline{u_n(t)} \big)  \partial_t\psi_j(t)  \mathrm{d}x.
\end{align*}
By assumption \ref{A1}-\ref{A2}, we have 
\begin{align*}
    \abs{\int G(|u_n(t)|^2)\nabla\psi_j(t)\mathrm{d}x}&+\abs{\int|u_n(t)|^2g(\abs{u_n(t)}^2)\nabla\psi_j(t)}\mathrm{d}x\\&\lesssim \int \abs{u_n(t)}^{2}\abs{\nabla \psi_j(t)}\mathrm{d}x+\int \abs{u_n(t)}^{2\sigma+2}\abs{\nabla \psi_j(t)}\mathrm{d}x.
\end{align*}
Thanks to Lemma \ref{der}, we have
$$
\left| \frac{\rm d}{\mathrm{d}t} \mathcal{J}_j(t, u_n(t)) \right|
\lesssim \int \big( |\nabla u_n(t)|^2 + |u_n(t)|^2 +\abs{u_n(t)}^{2\sigma+2}\big)  |\nabla \psi_j(t)|  \mathrm{d}x
+ \int |u_n(t)|^2  \| D^3_{xxx} \psi_j(t) \|  \mathrm{d}x.
$$
Now, remark that
\begin{align*}
\int \big( |\nabla u_n(t)|^2 &+ |u_n(t)|^2 +\abs{u_n(t)}^{2\sigma+2}\big)  |\nabla \psi_j(t)|  \mathrm{d}x
\\&\lesssim \int \big( |\nabla Q(t)|^2 + |Q(t)|^2 +\abs{Q(t)}^{2\sigma+2}\big)  |\nabla \psi_j(t)|  \mathrm{d}x+  \| w_n(t) \|_{H^1}^2+\norm{w_n(t)}_{H^1}^{2\sigma+2},
\end{align*}
and $$\int |u_n(t)|^2  \| D^3_{xxx} \psi_j(t) \|  \mathrm{d}x
\lesssim \int |Q(t)|^2  \| D^3_{xxx} \psi_j(t) \|  \mathrm{d}x
+  \| w_n(t) \|_{L^2(\mathbb{R})}^2 .$$
By the assumption \eqref{as} we know that for $t\in [\Upsilon^\dagger,T_n]$, we have $$\norm{w_n(t)}_{H^1}\lesssim e^{-\frac{(v_\star (t-\Upsilon_1))^2}{8}}.$$
Moreover, thanks to Lemma \ref{est2}, we have $$
\int \left( |\nabla Q(t)|^2 + |Q(t)|^2+\abs{Q(t)}^{2\sigma+2} \right) |\nabla \psi_j(t)|  \mathrm{d}x
\lesssim e^{-\frac{(v_\star (t-\Upsilon_1))^2}{4}},
$$
and $$\int |Q(t)|^2  \| D^3_{xxx} \psi_j(t) \|  \mathrm{d}x
\lesssim e^{-\frac{(v_\star (t-\Upsilon_1))^2}{4}}.$$
Finally, we combine all the estimations we obtained. \begin{equation*}
    \left| \frac{\rm d}{\mathrm{d}t} M_j(t,u_n(t)) \right|+\left| \frac{\rm d}{\mathrm{d}t} \mathcal{J}_j(t, u_n(t)) \right|\lesssim e^{-\frac{(v_\star (t-\Upsilon_1))^2}{4}}.\qedhere
\end{equation*}
\end{proof}
\begin{corollary}\label{44}
    For $n$ large enough and $t \in [\Upsilon^\dagger, T_n ]$, there holds
$$
\abs{S^{\mathrm{\rm loc}}(t, u_n(t)) - S^{\mathrm{\rm loc}}(T_n, Q(T_n))}
\lesssim \Upsilon_*^{-1} e^{-\frac{(v_\star (t-\Upsilon_1))^2}{4}}.
$$
\end{corollary}
\begin{proof}
    We can estimate this difference thanks to the previous estimate and Lemma \ref{gausse}: \begin{equation*}
        \abs{S^{\text{\rm loc}}(t, u_n(t)) - S^{\text{\rm loc}}(T_n, Q(T_n))}\leq \int _t^{T_n}\abs{\frac{\mathrm{d} \tilde{S}}{\mathrm{d}s}}\mathrm{d}s\lesssim \Upsilon_*^{-1}e^{-\frac{(v_\star (t-\Upsilon_1))^2}{4}}.\qedhere
    \end{equation*}
\end{proof}
Once again, thanks to Lemma \ref{011}, we obtain the estimate.
\begin{lemma}\label{55}
For all $t\in [\Upsilon^\dagger,T_n]$ and each $j\in\{1,...,N\}$ we have,
    $$
\abs{S_0^{\mathrm{\rm loc}}(t, Q(t))}
\lesssim \Upsilon_*^{-1} e^{-\frac{(v_\star (t-\Upsilon_1))^2}{4}},
$$
$$
\abs{S_j^{\mathrm{\rm loc}}(t, Q(t)) - S_j^{\mathrm{\rm loc}}(t, Q_j(t))}
\lesssim \Upsilon_*^{-1} e^{-\frac{(v_\star (t-\Upsilon_1))^2}{4}}.
$$
\end{lemma}
\begin{proof}
    We have: $$
\begin{aligned}
S_j^{\mathrm{\rm loc}}(t, Q(t)) - S_j^{\mathrm{\rm loc}}(Q_j(t)) 
&= \frac{1}{2} \Biggl(
\int \lvert \nabla Q(t) \rvert^2  \psi_j(t)  \mathrm{d}x 
- \int \lvert \nabla Q_j(t) \rvert^2  \psi_j(t)  \mathrm{d}x
\Biggr) \\
&\quad -  \frac{1}{2}\Biggl(
\int \lvert Q(t) \rvert^2 \log \lvert Q(t) \rvert^2  \psi_j(t)  \mathrm{d}x 
- \int \lvert Q_j(t) \rvert^2 \log \lvert Q_j(t) \rvert^2  \psi_j(t)  \mathrm{d}x
\Biggr) \\
&\quad - \Biggl( \int G(|Q(t)|^2)\psi_j(t)\mathrm{d}x-\int G(|Q_j(t)|^2)\psi_j(t)\mathrm{d}x\Biggr) \\
&\quad + \left(  1+\omega_j + \frac{|v_j|^2}{4} \right)
\left( M_j(t, Q(t)) - M_j(t, Q_j(t)) \right) \\
&\quad + v_j \cdot \left( \mathcal{J}_j(t, Q(t)) - \mathcal{J}_j(t, Q_j(t)) \right).
\end{aligned}
$$

$\circ$ \textit{The first term:}
$$|\nabla Q(t,x)|^{2}
= |\nabla Q_j(t,x)|^{2}
+ 2\Re\!\left( \nabla Q_j(t,x) \overline{\sum_{k\neq j} \nabla Q_k(t,x)} \right)
+ \left| \sum_{k\neq j} \nabla Q_k(t,x) \right|^{2}.$$
Thus, \begin{align*}
    \frac{1}{2}\abs{
\int \lvert \nabla Q(t) \rvert^2  \psi_j(t)  \mathrm{d}x 
- \int \lvert \nabla Q_j(t) \rvert^2  \psi_j(t)  \mathrm{d}x
}&\leq \sum_{k\neq j}\int \abs{\nabla Q_j(t)}\abs{\nabla Q_k(t)}\psi_j(t)\mathrm{d}x\\&\quad +\frac{N-1}{2}\sum_{k\neq j}\int \abs{ \nabla Q_k}^2\mathrm{d}x.
\end{align*}

$\circ$ \textit{The second term:} Let us normalize $Q$ and $Q_j$ in order to apply Lemma~\ref{liplog}.
  Set
$$
C_N := \bigl( N \max_j C_2^{(j)} \bigr)^{-1},
$$
where $C_2^{(j)}$ is the upper constant defined in Lemma \ref{011} associated with $Q_j$, so that
$$
C_N |Q(t,x)| \leq 1 \quad \text{ and } \quad C_N |Q_j(t,x)| \leq 1.
$$
Now define
$$
\widetilde Q(t,x) := C_N Q(t,x),
\qquad
\widetilde Q_j(t,x) := C_N Q_j(t,x) .
$$
And we have 
\begin{align*}
\Bigl| |Q(t,x)|^2 \log |Q(t,x)|^2 &-  |Q_j(t,x)|^2 \log |Q_j(t,x)|^2 \Bigr|
\\&=C_N^{-2}\Bigl| |\tilde{Q}(t,x)|^2 \big(\log |\tilde{Q}(t,x)|^2+\log C_N^{-2}\big) -  |\tilde{Q}_j(t,x)|^2 \big(\log |\tilde{Q}_j(t,x)|^2+\log C_N^{-2}\big) \Bigr|
\\&\lesssim
\abs{\abs{\tilde{Q}(t,x)}^2-\abs{\tilde{Q}_j(t,x)}^2}  \abs{1 - \log |\tilde{Q}_j(t,x)|^2} \\
&\lesssim \sum_{m\neq j}\abs{Q_m(t,x)}\bigg(\sum_k\abs{Q_k(t,x)}+\abs{Q_j(t,x)}\bigg)\bigg(1+\abs{x-x_j^*(t)}^2\bigg)\\
&\lesssim \sum_{m\neq j}\sum_{k=1}^N\abs{Q_m(t,x)}\abs{Q_k(t,x)}\bigg(1+\abs{x-x_j^*(t)}^2\bigg)\\ 
&=\sum_{k=1}^N\sum_{m\notin\{j,k\}}^N\abs{Q_m(t,x)}\abs{Q_k(t,x)}\bigg(1+\abs{x-x_j^*(t)}^2\bigg)\\& \quad+\sum_{k\neq j}^N \abs{Q_k(t,x)}^2\bigg(1+\abs{x-x_j^*(t)}^2\bigg).
\end{align*}
The first inequality is a consequence of Lemma~\ref{liplog}. Taking the $L^2(\mathbb{R}^d)$ inner product with $\psi_j(t,x)$, we obtain
\begin{align*}
    \int\Bigl| |Q(t)|^2 \log |Q(t)|^2 &-  |Q_j(t)|^2 \log |Q_j(t)|^2 \Bigr|\psi_j(t)\mathrm{d}x\\&\lesssim \sum_{k=1}^N\sum_{m\notin\{j,k\}}^N\int\abs{Q_m(t)}\abs{Q_k(t)}\bigg(1+\abs{x-x_j^*(t)}^2\bigg)\psi_j(t)\mathrm{d}x\\& \quad+\sum_{k\neq j}^N \int \abs{Q_k(t)}^2\bigg(1+\abs{x-x_j^*(t)}^2\bigg)\psi_j(t)\mathrm{d}x
    \\&
    \lesssim \Upsilon_*^{-1} e^{-\frac{(v_\star (t-\Upsilon_1))^2}{4}}.
\end{align*}
In the last inequality, we used Lemma \ref{est4} for the first term. For the second term, we have, for $k \neq j$,
\begin{align*}
    \int \abs{Q_k(t)}^2\Big(1+\abs{x-x_j^*(t)}^2\Big)\psi_j(t)\mathrm{d}x
    &\lesssim \int \abs{Q_k(t)}^2\Big(1+\abs{x-x_k^*(t)}^2\Big)\psi_j(t)\mathrm{d}x \\
    &\qquad + \abs{x_k^*(t)-x_j^*(t)}^2 \int \abs{Q_k(t)}^2\psi_j(t)\mathrm{d}x.
\end{align*}
which yields the desired conclusion by applying Lemma~\ref{est2}.

$\circ$ \textit{The third term:}
 By Lemma \ref{power}, we have 
 \begin{equation}
     \abs{G(|Q(t)|^2)-G(|Q_j(t)|^2)}\lesssim  \abs{\lvert Q(t) \rvert^{2}-\abs{Q_j(t)}^{2}}+ \abs{\lvert Q(t) \rvert^{2\sigma+2}-\abs{Q_j(t)}^{2\sigma+2}}.
 \end{equation}
We apply the mean value theorem to the function $x \mapsto x^{\sigma+1}$ and use the fact that $Q_j(t)$ and $Q(t)$ are bounded in $L^{\infty}$ uniformly in time to obtain
 $$\int\abs{G(|Q(t)|^2)-G(|Q_j(t)|^2)}\psi_j(t)\mathrm{d}x\lesssim \int \abs{\lvert Q(t) \rvert^2-\abs{Q_j(t)}^2} \psi_j(t)  \mathrm{d}x.$$
On the other hand, we have $$|Q(t,x)|^{2}
= |Q_j(t,x)|^{2}
+ 2\Re\!\left( Q_j(t,x) \overline{\sum_{k\neq j} Q_k(t,x)} \right)
+ \left| \sum_{k\neq j} Q_k(t,x) \right|^{2}.$$
Hence, $$
\int\abs{G(|Q(t)|^2)-G(|Q_j(t)|^2)}\psi_j(t)\mathrm{d}x
\lesssim \sum_{k\neq j}\int \abs{Q_j(t)}\abs{Q_k(t)}\psi_j(t)\mathrm{d}x+(N-1)\sum_{k\neq j}\int \abs{Q_k(t)}^2\psi_j(t)\mathrm{d}x.
$$
For the first term, we use Lemma~\ref{est4}, for the second one, we use Lemma~\ref{est2}. We conclude that
$$
\int\abs{G(|Q(t)|^2)-G(|Q_j(t)|^2)}\psi_j(t)\mathrm{d}x
\lesssim \Upsilon_*^{-1} e^{-\frac{(v_\star (t-\Upsilon_1))^2}{4}}.
$$

$\circ$\textit{ The last two term:} We have $$M_j(Q(t))=M_j(Q_j(t))+\int\Re(Q_j(t)\overline{\sum_{k\neq j}Q_k(t)})\psi_j(t)\mathrm{d}x+\frac12\int \abs{\sum_{k\neq j}Q_k(t)}^2\psi_j(t)\mathrm{d}x,$$
$$\mathcal{J}_j(Q(t))-\mathcal{J}_j(Q_j(t))=\frac12\sum_{(k,m)\neq (j,j)}\Im \int \nabla Q_k(t)\overline{Q_m(t)}\psi_j(t)\mathrm{d}x.$$
The conclusion comes in the same way. This completes the proof for the second inequality.

For the first inequality, we observe that $\psi_0(t)\leq 1-\psi_\ell(t)$ for all $\ell\in \{1,...,N\}$, then we repeat a similar computation, and for some $j\in \{1,...,N\}$, we obtain
$$
\abs{S_0^{\mathrm{\rm loc}}(t,Q(t)) - S_0^{\mathrm{\rm loc}}(t,Q_j(t))}
\lesssim \Upsilon_*^{-1} e^{-\frac{(v_\star (t - \Upsilon_1))^2}{4}}.
$$
Thanks again to Lemma~\ref{est1} and Lemma \ref{est2}, we get \begin{equation*}
    \abs{S_0^{\mathrm{\rm loc}}(t,Q_j(t))}
\lesssim \Upsilon_*^{-1} e^{-\frac{(v_\star (t - \Upsilon_1))^2}{4}},
\end{equation*} thus
\begin{equation*}
    \abs{S_0^{\mathrm{\rm loc}}(t,Q(t))}
\lesssim \Upsilon_*^{-1} e^{-\frac{(v_\star (t - \Upsilon_1))^2}{4}}.\qedhere
\end{equation*}
\end{proof}
\begin{proof}[Proof of Porposition \ref{slow}]
    We conclude the proof using the decomposition in \eqref{10}. Each term is then estimated by applying Corollary \ref{44}, Lemma \ref{55}, and Lemma \ref{bot33}.
\end{proof}
\subsection{Proof of the Bootstrap Property }
Now we prove Lemma~\ref{05}. Let $t \in [\Upsilon^\dagger, T_n]$.  
Thanks to Lemma~\ref{bot2}, we have
$$
\|\nabla w_n(t)\|_{2}^2 \lesssim 
S^{\mathrm{\rm loc}}(t,u_n(t)) - \sum_{j=1}^N S_j(Q_j)
+ C \Upsilon_*^{-1/2} \exp\!\left(-\frac{(v_\star (t-\Upsilon_1))^2}{4}\right).
$$
Applying Lemma~\ref{slow}, we obtain
$$
\|\nabla w_n(t)\|_{2}^2 \lesssim 
\Upsilon_*^{-1/2} \exp\!\left(-\frac{(v_\star (t-\Upsilon_1))^2}{4}\right).
$$
Finally, thanks to Lemma~\ref{bot1}, we conclude that
$$
\|w_n(t)\|_{H^1}^2 \lesssim 
\Upsilon_*^{-1/2} \exp\!\left(-\frac{(v_\star (t-\Upsilon_1))^2}{4}\right).
$$
Choosing $\Upsilon_*$ sufficiently large yields the desired conclusion.
\section{Proof of the compactness property}\label{5}
This subsection is devoted to the proof of the compactness property i.e Lemma \ref{07}.
The proof is completely similar to that in~\cite{25,22,12,15}.
The main idea is to show that $u_n(\Upsilon_0)$ is uniformly bounded in $W_1(\mathbb{R}^d)$.

\begin{proposition}\label{cv}
    The initial data $u_n(\Upsilon_0)$ is uniformly bounded in $W_1(\mathbb{R}^d)$. In particular, up to the extraction of a subsequence, there exists $u_0 \in W_1(\mathbb{R}^d)$ such that
$$
u_n(\Upsilon_0) \rightharpoonup u_0
\quad \text{weakly in } W_1(\mathbb{R}^d).
$$
\end{proposition}
\begin{proof}
Since $ u_n(\Upsilon_0) $ is uniformly bounded in $ H^1(\mathbb{R}^d) $, it remains only to prove that the logarithmic term
$$
\int_{\mathbb{R}^d} |u_n(\Upsilon_0)|^2  \bigl| \log|u_n(\Upsilon_0)|^2 \bigr|  \mathrm{d}x,
$$
is uniformly bounded with respect to $ n $.

We first remark that
$$
E(u_n(t)) = E(Q(T_n)).
$$
Thanks to Lemma~\ref{011} and Corollary \ref{0111}, it follows that the energy $ E(u_n) $ is uniformly bounded.

Recall that the energy is given by
$$
E(u_n)
= \frac{1}{2} \int |\nabla u_n|^2  \mathrm{d}x
+ \frac{1}{2} \int |u_n|^2  \mathrm{d}x
- \frac{1}{2} \int |u_n|^2 \log|u_n|^2  \mathrm{d}x
-  \int G(|u_n|^2)  \mathrm{d}x .
$$
We decompose the logarithmic term into two regions:
\begin{align*}
    \int |u_n(\Upsilon_0)|^2 \bigl| \log|u_n(\Upsilon_0)|^2 \bigr|\mathrm{d}x
&= \int_{|u_n(\Upsilon_0)|^2 \le 1} |u_n(\Upsilon_0)|^2 \bigl| \log|u_n(\Upsilon_0)|^2 \bigr|\mathrm{d}x
\\&\quad + \int_{|u_n(\Upsilon_0)|^2 > 1} |u_n(\Upsilon_0)|^2 \log|u_n(\Upsilon_0)|^2\mathrm{d}x.
\end{align*}
From the definition of the energy, we obtain
\begin{align*}
\frac12\int_{|u_n(\Upsilon_0)|^2 \le 1} |u_n(\Upsilon_0)|^2 \bigl| \log|u_n(\Upsilon_0)|^2 \bigr|
&= E(u_n(\Upsilon_0))
- \frac{1}{2} \|u_n(\Upsilon_0)\|_{H^1}^2 + \int G(|u_n(\Upsilon_0)|^2)  \mathrm{d}x  \\
&\quad + \frac12\int_{|u_n(\Upsilon_0)|^2 > 1} |u_n(\Upsilon_0)|^2 \log|u_n(\Upsilon_0)|^2\mathrm{d}x.
\end{align*}
By Lemma \ref{power} and the Sobolev embedding $ H^1(\mathbb{R}^d) \hookrightarrow L^{2\sigma+2}(\mathbb{R}^d) $, the third term is uniformly bounded.

For the last term, recall that for any $ \delta > 0 $ and any $ y > 1 $,
$$
\log(y) \le \frac{1}{\delta}  y^{\delta}.
$$
Therefore,
$$
\int_{|u_n(\Upsilon_0)|^2 > 1} |u_n(\Upsilon_0)|^2 \log|u_n(\Upsilon_0)|^2\mathrm{d}x
\lesssim \int |u_n(\Upsilon_0)|^{2+2\delta}  \mathrm{d}x.
$$
Choosing $ \delta > 0 $ sufficiently small, the Sobolev embedding yields
$$
\int |u_n(\Upsilon_0)|^{2+2\delta}  \mathrm{d}x
\lesssim \|u_n(\Upsilon_0)\|_{H^1}^{2+2\delta}.
$$
Since $ \|u_n(\Upsilon_0)\|_{H^1} $ is uniformly bounded, we conclude that
$$
\int_{\mathbb{R}^d} |u_n(\Upsilon_0)|^2  \bigl| \log|u_n(\Upsilon_0)|^2 \bigr|  \mathrm{d}x,
$$
is uniformly bounded with respect to $ n $, Thus, $u_n(\Upsilon_0)$ is bounded in $W_1(\mathbb{R}^d)$ uniformly in $n$.
\end{proof}

\begin{proof}[Proof of Proposition \ref{07}] The proof of the proposition follows exactly the same arguments as in \cite{12,15,22,25}. Since $u_n(\Upsilon_0)$ is uniformly bounded in $H^1(\mathbb{R}^d)$ , we only have to prove compactness at infinity in $L^2(\mathbb{R}^d)$. 

Choose $\delta>0$ and let $T_\delta \ge \Upsilon_0$ be such that
$$
e^{-\frac{(v_\star T_\delta)^2}{4}} \le \frac{\delta}{4}.
$$
Thanks to the assumption \eqref{as}, for any $n \in \mathbb{N}$ we have
$$
\| u_n(T_\delta) - Q(T_\delta) \|_{L^2}^2 \le \frac{\delta}{4}.
$$
By Lemma \ref{011}, we can find $\bar r_\delta > 0$ such that
$$
\int_{|x| > \bar r_\delta} |Q(T_\delta,x)|^2  \mathrm{d}x < \frac{\delta}{4}.
$$
Combining these two estimates, we get, for all $n \in \mathbb{N}$, 
$$
\int_{|x| > \bar r_\delta} |u_n(T_\delta, x)|^2  \mathrm{d}x < \frac{\delta}{2}.
$$
We want to show that there exists $r_\delta > 0$ such that, for any $n \in \mathbb{N}$,
$$
\int_{|x| > r_\delta} |u_n(\Upsilon_0, x)|^2  \mathrm{d}x < \delta.
$$
To this end, we transfer the property from $T_\delta$ to $\Upsilon_0$. Let $\hat r_\delta > 0$ be fixed later, and let $\rho \in \mathcal{C}^1(\mathbb{R}, \mathbb{R})$ be a cut-off function satisfying
$$
\rho(s) = 0 \text{ for } s < 0, \quad 
\rho(s) = 1 \text{ for } s > 1, \quad 
0 \le \rho(s) \le 1 \text{ for all } s \in \mathbb{R}.
$$
Now, define
$$
R(t) := \int |u_n(t, x)|^2  \rho\bigg( \frac{|x| - \bar r_\delta}{\hat r_\delta} \bigg)  \mathrm{d}x.
$$
After straightforward calculations, we have
$$
R'(t) = \frac{2}{\hat r_\delta} \int 
\Im \Bigl( \overline{u_n(t, x)}  \frac{x}{|x|} \cdot \nabla u_n(t, x) \Bigr) 
 \rho'\Bigl( \frac{|x| - \bar r_\delta}{\hat r_\delta} \Bigr)  \mathrm{d}x.
$$
Since $\| u_n(t) \|_{H^1(\mathbb{R}^d)}$ is bounded independently of $n$ and $t$, we define
$$
C_0 := \sup_{n, t} \| u_n(t) \|_{H^1(\mathbb{R}^d)}^2<+\infty,
$$
so that
$$
|R'(t)| \le \frac{2 C_0}{\hat r_\delta}.
$$
Now choose $\hat r_\delta$ such that
$$
\frac{2C_0}{\hat r_\delta} T_\delta < \frac{\delta}{2}.
$$
Then, by integrating from $\Upsilon_0$ to $T_\delta$, we obtain
$$
R(\Upsilon_0) - R(T_\delta) \le \frac{\delta}{2}.
$$
However,
$$
R(T_\delta) = \int |u_n(T_\delta, x)|^2  \rho\Bigl( \frac{|x| - \bar r_\delta}{\hat r_\delta} \Bigr)  \mathrm{d}x \le \frac{\delta}{2}.
$$
It follows that
$$
R(\Upsilon_0) \le \delta.
$$
Finally, we set
$$
r_\delta := \bar r_\delta + \hat r_\delta.
$$
By the definition of $\rho$, we conclude that, for all $n \in \mathbb{N}$,
$$
\int_{|x| > r_\delta} |u_n(\Upsilon_0, x)|^2  \mathrm{d}x \le R(\Upsilon_0) \le \delta.
$$
This proves the compactness argument in $L^2(\mathbb{R}^d)$: there exists $u_0 \in H^1(\mathbb{R}^d)$ such that (up to a subsequence)
$$
u_n(\Upsilon_0) \to u_0 \quad \text{in } L^2(\mathbb{R}^d) \text{ as } n \to \infty.
$$
Moreover, by Proposition \ref{cv}, we have $u_0 \in W_1(\mathbb{R}^d)$.
\end{proof}
\section*{Acknowledgments}
The authors acknowledge the support of the CDP C2EMPI, as well as the French State under the France-2030 programme, the University of Lille, the Initiative of Excellence of the University of Lille, and the European Metropolis of Lille for their funding and support of the R-CDP-24-004-C2EMPI project.
I would like to express my deepest gratitude to my supervisors, Vianney Combet, Guillaume Ferriere, and Sahbi Keraani, for their invaluable guidance, availability, and support throughout the completion of this work.

\renewcommand{\appendixpagename}{Appendix}
\renewcommand{\appendixtocname}{Appendix}

\appendix
\appendixpage
\addappheadtotoc
\setcounter{theorem}{0}

In this appendix, the first section is devoted to the interaction between the solitons  (associated with the power-type nonlinearity and the log-type nonlinearity), 
while the second section is devoted to the proof of Lemma~\ref{08}.
\section{Estimates of the interaction between the solitons}
To estimate the interaction between the solitons associated with the subcritical nonlinearity we use Lamma \ref{est1} and Lemma \ref{est2}. Before stating the result, we recall the following inequality.
\begin{lemma}\cite[p.84]{6}\label{power}
For all $z_1, z_2 \in \mathbb{C}$ and $\alpha \ge 0$, we have
$$
| z_1 |z_1|^{\alpha} - z_2 |z_2|^{\alpha} | \le (\alpha + 1) \big( |z_1|^{\alpha} + |z_2|^{\alpha} \big) |z_1 - z_2|.
$$
Moreover,
$$\abs{(z_1g(\abs{z_1}^{2})-z_2g(\abs{z_2}^{2}))\overline{(z_1 - z_2)} }\lesssim (\abs{z_2}^{2\sigma}+1)\abs{z_1-z_2}^2+\abs{z_1-z_2}^{2\sigma+2},$$

$$
| z_1 g(|z_1|^2) - z_2 g(|z_2|^2) | \lesssim \big( 1+|z_1|^{2\sigma} + |z_2|^{2\sigma} \big) |z_1 - z_2|,
$$ 
and 
$$\abs{G(|z_1|^2)-G(|z_2|^2)}\lesssim \abs{\abs{z_1}^2-\abs{z_2}^2}+\abs{\abs{z_1}^{2\sigma+2}-\abs{z_2}^{2\sigma+2}}.$$
\end{lemma}

\begin{lemma}\label{est5}
    For $t$ sufficiently large, we have
$$
\bigl\| Q(t)g(|Q(t)|^{2}) - \sum_{j=1}^N Q_j(t)g(|Q_j(t)|^{2}) \bigr\|_{L^2} 
\lesssim \exp\!\Biggl( -\frac{(v_\star (t-\Upsilon_1))^2}{8} \Biggr).
$$
\end{lemma}
\begin{proof}
    Let $k \in \mathbb{N}^*$. Then, applying Lemma~\ref{power}, we obtain
\begin{align*}
    \bigl| Q(t,x) g(|Q(t,x)|^{2}) &- \sum_{j=1}^N Q_j(t,x) g(|Q_j(t,x)|^{2}) \bigr|
\\&\leq\abs{Q(t,x) g(|Q(t,x)|^{2})-Q_k(t,x) g(|Q_k(t,x)|^{2})}+\sum_{j\neq k}\abs{Q_j(t,x)g(\abs{Q_j(t,x)}^2)}
\\&\lesssim \sum_{j \neq k} |Q_j(t,x)| \Bigl( 1+|Q(t,x)|^{2\sigma} + |Q_k(t,x)|^{2\sigma} + \abs{g(|Q_j(t,x)|^{2})} \Bigr)
\\&\lesssim \sum_{j \neq k} |Q_j(t,x)|.
\end{align*}
Hence by Lemma \ref{est2},
$$
\bigl\| \psi_k(t) \bigl(Q(t) g(|Q(t)|^{2}) - \sum_{j=1}^N Q_j(t) g(|Q_j(t)|^{2}) \bigr) \bigr\|_{L^2}
\le \sum_{j \neq k} \| \psi_k(t) |Q_j(t)| \|_{L^2}
\lesssim \exp\!\Biggl( -\frac{(v_\star (t-\Upsilon_1))^2}{8} \Biggr).
$$
Using $\sum_k \psi_k(t) = 1$, we deduce that
\begin{equation*}
    \bigl\| Q(t) g(|Q(t)|^{2}) - \sum_{j=1}^N Q_j(t) g(|Q_j(t)|^{2}) \bigr\|_{L^2}
\lesssim \exp\!\Biggl( -\frac{(v_\star (t-\Upsilon_1))^2}{8} \Biggr).\qedhere
\end{equation*}
\end{proof}
\begin{lemma}\cite[Lemma 3.3]{16}\label{liplog}
For all $z_1,z_2\in \mathbb{C}$ such that $|z_1|, |z_2| \le 1$ and $z_1\neq 0$, we have
    $$
\bigl| z_1\log |z_1| - z_2\log |z_2| \bigr| \lesssim (1 - \log |z_1|) |z_1 - z_2|.
$$
\end{lemma}
In \cite{15}, the previous lemma is used, together with the explicit Gausson, to obtain interaction estimates for the logarithmic nonlinearity. In our case, we do not have an explicit soliton, however, by using the lower bound obtained in Lemma~\ref{011}, we can derive the same estimate for the interactions arising from the logarithmic nonlinearity. 
\begin{lemma}\label{logest}For any $j\in\{1,...,N\}$ and all $x\in \mathbb{R}^d$ as in $t\to +\infty$, we have 
     $$\abs{Q(t,x)\log\abs{Q(t,x)}^2-\sum_{k=1}^NQ_k(t,x)\log\abs{Q_k(t,x)}^2} \lesssim \sum_{k\neq j}^N|Q_k(t,x)|\abs{1+(t-\Upsilon_1)^2+\abs{x-x_k^*(t)}^2}.$$
\end{lemma}
\begin{proof}
   Let us normalize $Q$ and $Q_j$ in order to apply Lemma~\ref{liplog}. Set
$$
C_N := \bigl( N \max_j C_2^{(j)} \bigr)^{-1},
$$
where $C_2^{(j)}$ is the upper constant defined in Lemma \ref{011} associated with $Q_j$, so that
$$
C_N |Q(t,x)| \leq 1 \quad \text{ and } \quad C_N |Q_j(t,x)| \leq 1.
$$
Now define
$$
\widetilde Q(t,x) := C_N Q(t,x),
\qquad
\widetilde Q_j(t,x) := C_N Q_j(t,x) .
$$
And we have 
$$
C_N \Bigl| Q(t,x) \log |Q(t,x)|^2 - \sum_{k=1}^N Q_k(t,x) \log |Q_k(t,x)|^2 \Bigr|
=
\Bigl| \widetilde Q(t,x) \log |\widetilde Q(t,x)|^2
- \sum_{k=1}^N \widetilde Q_k(t,x) \log |\widetilde Q_k(t,x)|^2 \Bigr|.
$$
We can apply Lemma~\ref{liplog}, and we obtain
\begin{align*}
\Bigl| Q(t,x) \log |Q(t,x)|^2 &- \sum_{k=1}^N Q_k(t,x) \log |Q_k(t,x)|^2 \Bigr|
\\&= C_N^{-1}\Bigl| \tilde{Q}(t,x) \log |\tilde{Q}(t,x)|^2 - \sum_{k=1}^N \tilde{Q}_k(t,x) \log |\tilde{Q}_k(t,x)|^2 \Bigr|
\\
&\lesssim
\bigl| \tilde{Q}(t,x) \log |\tilde{Q}(t,x)|^2 - \tilde{Q}_j(t,x) \log |\tilde{Q}_j(t,x)|^2 \bigr|
+ \sum_{k \neq j} |\tilde{Q}_k(t,x)|  |\log |\tilde{Q}_k(t,x)|^2| \\
&\lesssim
\sum_{k \neq j} |Q_k(t,x)|  |1 - \log |Q_j(t,x)|^2|
+  \sum_{k \neq j} |Q_k(t,x)|\Big(|x - x_k^*(t)|^2+1\Big) \\
&\lesssim
\sum_{k \neq j} |Q_k(t,x)|
\Big( 1 + |x - x_k^*(t)|^2 + |x - x_j^*(t)|^2 \Big) \\
&\lesssim
\sum_{k \neq j} |Q_k(t,x)|
\Big( 1 + (t-\Upsilon_1)^2 + |x - x_k^*(t)|^2 \Big).
\end{align*}
In the last inequality, we used the estimate
\begin{equation*}
     |x - x_j^*(t)|
\le |x - x_k^*(t)| + |x_j^*(t) - x_k^*(t)|
\le |x - x_k^*(t)| + |x_j - x_k| + |v_j - v_k| t.\qedhere
\end{equation*}
\end{proof}
\begin{lemma}\label{est6}
For $t$ sufficiently large, we have
$$
\bigl\| Q(t) \log|Q(t)|^2 - \sum_{j=1}^N Q_j(t) \log|Q_j(t)|^2 \bigr\|_{L^2}
\lesssim \exp\!\Biggl( -\frac{(v_\star (t-\Upsilon_1))^2}{8} \Biggr).
$$
\end{lemma}
\begin{proof}
    Thanks to Lemma \ref{logest}
    \begin{align*}
        \abs{Q(t,x)\log\abs{Q(t,x)}^2-\sum_{k=1}^NQ_k(t,x)\log\abs{Q_k(t,x)}^2} \lesssim \sum_{k\neq j}^N|Q_k(t,x)|\abs{1+(t-\Upsilon_1)^2+\abs{x-x_k^*(t)}^2}.
    \end{align*}
    Taking the $L^2(\psi_k(t)\mathrm{d}x)$-norm, we obtain
    \begin{align*}
        \norm{Q(t)\log\abs{Q(t)}^2-\sum_{k=1}^NQ_k(t)\log\abs{Q_k(t)}^2}_{L^2(\psi_k(t)\mathrm{d}x)}&\lesssim(1+(t-\Upsilon_1)^2) \sum_{k\neq j}^N\norm{Q_k(t)}_{L^{2}(\psi_k(t))\mathrm{d}x}\\&\quad+\sum_{k\neq j}^N \norm{\abs{x-x_k^*(t)}^2Q_k(t)}_{L^2(\psi_k(t)\mathrm{d}x)}.
    \end{align*}
The first term is estimated by the first estimate in Lemma~\ref{est2}, while the second is controlled by the last estimate in Lemma~\ref{est1}, and using $\sum_{k=0}^N \psi_k(t) = 1$ we get the conclusion.
\end{proof}
\section{Proof of Lemma \ref{est1} and Lemma \ref{est2}}
\begin{proof}[Proof of Lemma \ref{est1}]
By the definition of partition of unity, we have
$$
\psi_j(t,x) = 1 \quad \text{for } x \in B_j(t):=B\bigg(x_j^*(t), \frac{v_\star (t-\Upsilon_1)}{2} + 1\bigg).
$$
Therefore, the first quantity in the first estimate can be controlled outside $B_j(t)$.\\
Then, setting
$
B_0(t) := B(0, \frac{v_\star (t-\Upsilon_1)}{2} + 1),
$
we can easily compute for all $t\geq \Upsilon_1+\Upsilon_*$ and $\Upsilon_*$ large enough
\begin{align*}
\int_{B_j(t)} |Q_j(t)|^2 (1 - \psi_j(t))  \mathrm{d}x 
&= \int_{B_j^c(t)} |Q_j(t)|^2  \mathrm{d}x \\
&\lesssim \int_{B_j^c(t)} \exp\Bigl(-(|x - x_j^*(t)|-r_0)^2\Bigr)  \mathrm{d}x \\
&\lesssim \int_{B_0^c(t)} \exp\Bigl(-(|y|-r_0)^2\Bigr)  dy \\
&\lesssim \bigg(\frac{v_\star (t-\Upsilon_1)}{2} + 1\bigg)^{d-2}e^{-v_\star(t-\Upsilon_1)}\exp\Bigl(-\frac{(v_\star (t - \Upsilon_1))^2}{4}\Bigr)\\
&\lesssim \Upsilon_*^{-6} \exp\Bigl(-\frac{(v_\star (t - \Upsilon_1))^2}{4}\Bigr).
\end{align*}
In the last inequality, we used Lemma \ref{gausse} with $r=\frac{v_\star (t-\Upsilon_1)}{2} + 1$. For the last two terms, we just need to observe that
$
\partial_x \psi_k(t,x) = - \partial_x \bigl(1 - \psi_k(t,x)\bigr),
$
which is why the case $k = j$ is also included. 
\end{proof}

\begin{proof}[Proof of Lemma \ref{est2}]
The first two estimates can be deduced in the same way as in the proof of Lemma \ref{est1}. The last estimate can be deduced similarly to \cite{15}: instead of using an explicit expression for the soliton (such as the Gaussons in~\cite{15}), we only use the lower and upper bounds proved in Lemma~\ref{011}. Thus, we have \begin{equation}
    - \lvert x - x_j^*(t) \rvert^2
\lesssim
 \log \lvert Q_j(t) \rvert^2 
\lesssim
- (\lvert x - x_j^*(t) \rvert-r_0)^2.
\end{equation}
Observe that for $k \neq j$ we have $
\psi_k \le 1 - \psi_j
$. Then, 
\begin{align*}
    \bigl\| Q_j(t) \log |Q_j(t)|^2 \bigr\|_{L^2((1-\psi_j(t))\mathrm{d}x)}
+\bigl\| Q_j(t) \log |Q_j(t)|^2 \bigr\|_{L^2(\psi_k(t)\mathrm{d}x)}&\lesssim
\bigl\| Q_j(t) |x - x_j^*(t)|^2 \bigr\|_{L^2((1-\psi_j(t))\mathrm{d}x)}
\end{align*}
By Lemma \ref{est1}, we then obtain the desired conclusion.
\end{proof}

\setcounter{section}{2} 
\renewcommand{\thesection}{\arabic{section}}

\section*{Declaration of competing interest}
No competing interest.

\end{document}